\documentclass[final, 11pt, a4paper]{article}
\usepackage[a4paper, total={160mm,240mm}]{geometry}

\usepackage{xcolor}
\usepackage{amsmath,amssymb,amscd}
\usepackage[mathscr]{eucal}
\usepackage{mathtools}
\usepackage{stmaryrd}
\usepackage{url}
\usepackage[colorlinks = true,linkcolor = blue,urlcolor  = blue,citecolor = blue,anchorcolor = blue,linktocpage=true
]{hyperref}
\usepackage{amsthm}
\allowdisplaybreaks[2]
\usepackage[inline]{enumitem}
\setlist[enumerate,1]{label=(\arabic*),ref=\arabic*}
\setlist[itemize,1]{itemsep=0pt}
\usepackage{tabularray}
\UseTblrLibrary{amsmath}
\usepackage{paracol}
\usepackage{adjustbox}
\usepackage{caption}
\usepackage{afterpage}

\usepackage{tikz}
\usetikzlibrary{positioning}
\usetikzlibrary{cd}
\usetikzlibrary{shapes.geometric}
\usetikzlibrary{decorations.pathreplacing,decorations.markings}
\usetikzlibrary{backgrounds}
\usetikzlibrary{arrows.meta}

\newtheorem{theorem}{Theorem}[section]
\newtheorem{proposition}[theorem]{Proposition}
\newtheorem{conjecture}[theorem]{Conjecture}

\theoremstyle{definition}
\newtheorem{remark}[theorem]{Remark}

\theoremstyle{remark}
\newtheorem*{acknowledgment}{Acknowledgment}

\newcommand{\dedekind}[1]{\left(\!\!\left(#1\right)\!\!\right)}
\DeclarePairedDelimiter{\floor}{\lfloor}{\rfloor}

\DeclareMathOperator{\diag}{diag}

\DeclareMathOperator{\lcm}{lcm}

\DeclareMathOperator{\tr}{tr}

\DeclareMathOperator{\Li}{Li}
\DeclareMathOperator{\Gal}{Gal}
\DeclareMathOperator{\SL}{SL}

\title{Remarks on Nahm sums for symmetrizable matrices}
\author{Yuma Mizuno}
\AtEndDocument{\bigskip{\footnotesize%
  \textsc{Department of Mathematics and Informatics, Faculty of Science, Chiba University, Chiba 263-8522, Japan.} \par  
  \textit{E-mail address}, Y.~Mizuno: \texttt{ymizuno@math.s.chiba-u.ac.jp}
}}
\date{}

\begin{document}
\maketitle

\begin{abstract}
  Nahm sums are specific $q$-hypergeometric series associated with symmetric positive definite matrices.
  In this paper we study Nahm sums associated with symmetrizable matrices.
  We show that one direction of Nahm's conjecture, which was proven by Calegari, Garoufalidis, and Zagier for the symmetric case,
  also holds for the symmetrizable case.
  This asserts that the modularity of a Nahm sum implies that a certain element in a Bloch group associated with the Nahm sum is a torsion element.
  During the proof, we investigate the radial asymptotics of Nahm sums. 
  Finally, we provide lists of candidates of modular Nahm sums for symmetrizable matrices based on numerical experiments.
\end{abstract}

\section{Introduction}
The Rogers-Ramanujan identities relate infinite sum and infinite product expressions of $q$-series:
\begin{equation}\label{eq:RR sum}
  \sum_{n=0}^\infty \frac{q^{n^2}}{(q;q)_n} = \frac{1}{(q, q^4; q^5)_\infty},\quad
  \sum_{n=0}^\infty \frac{q^{n^2+n}}{(q;q)_n} = \frac{1}{(q^2, q^3; q^5)_\infty},
\end{equation}
where $(x;q)_n \coloneqq (1-x) (1-q x) \cdots (1-q^{n-1} x)$,
$(x;q)_\infty \coloneqq \prod_{i=0}^\infty (1 - q^i x)$, and
$(x_1,\dots,x_n;q)_\infty \coloneqq (x_1;q)_\infty \cdots (x_n;q)_\infty$.
One of the remarkable consequences of the Rogers-Ramanujan identities is that
the infinite sums in \eqref{eq:RR sum}, which are $q$-hypergeometric series, 
are also modular functions. Specifically, the vector-valued function on the upper-half plane
\begin{align*}
  g(\tau) = 
  \begin{pmatrix}
    q^{-1/60} \displaystyle{\sum_{n=0}^\infty \frac{q^{n^2}}{(q;q)_n}}, &
    q^{11/60} \displaystyle{\sum_{n=0}^\infty \frac{q^{n^2 + n}}{(q;q)_n}}
  \end{pmatrix}^\mathsf{T}
\end{align*}
with $q = e^{2 \pi i \tau}$ has the following transformation formulas:
\begin{align*}
  g(\tau + 1) = 
  \begin{pmatrix}
    \zeta_{60}^{-1} & 0 \\
    0 & \zeta_{60}^{11}
  \end{pmatrix}
  g(\tau), \quad
  g(-\frac{1}{\tau}) = 
  \frac{2}{\sqrt{5}}
  \begin{pmatrix}
    \sin \frac{2 \pi}{5} & \sin \frac{\pi}{5} \\
    \sin \frac{\pi}{5} & -\sin \frac{2 \pi}{5}
  \end{pmatrix}
  g(z)
\end{align*}
where $\zeta_N = e^{2 \pi i / N}$.

In the study of rational conformal field theories,
Nahm \cite{Nahm} considered a higher rank generalization of the infinite sums appearing in the Rogers-Ramanujan identities \eqref{eq:RR sum},
which we call \emph{Nahm sum}.
Let $N$ be a natural number, and suppose that $A \in \mathbb{Q}^{N \times N}$ is a symmetric positive definite matrix,
$b \in \mathbb{Q}^N$ is a vector, 
and $c \in \mathbb{Q}$ is a scalar.
The Nahm sum is defined by
\begin{equation}\label{eq:intro nahm}
  f_{A, b, c} (q) \coloneqq \sum_{n \in \mathbb{N}^N} 
    \frac{q^{ \frac{1}{2} n^\mathsf{T} A n + n^\mathsf{T} b + c}}{(q;q)_{n_1} \cdots (q;q)_{n_N}}.
\end{equation}
In general, Nahm sums rarely have modularity. 
However, as discussed in \cite{Nahm}, a character in rational conformal field theories is often expressed as a Nahm sum
and also possess modularity.
Motivated by such situations, Nahm conjectured that the modularity of 
Nahm sums are related to torsion elements of the Bloch group.
Let us explain the conjecture in a little more detail. The matrix $A$ determines the \emph{Nahm's equation} 
\begin{align}\label{eq:intro nahm equation}
  1 - z_i = \prod_{j=1}^N z_j^{A_{ij}}
  \quad \text{for $i = 1, \dots, N$},
\end{align}
and each solution of this equation determines an element in the Bloch group $B(F)$:
\begin{align}\label{eq:intro xi}
  \xi_{A}(z) \coloneqq \sum_{i=1}^N [z_i] ,
\end{align}
where $F$ is the number field generated by $z_1, \dots, z_N$.
The \emph{Nahm's conjecture} as stated in \cite{Zagier} is that, for any $A$, the Nahm sum $f_{A, b, c}(q)$ is modular 
for some $b, c$ if and only if $\xi_{A}(z)$ is torsion in $B(F)$ for any solution $z = (z_i)_i$ \cite{Nahm,Zagier}.
Although this conjecture itself is known to be false \cite{VlasenkoZwegers,Zagier},
Calegari, Garoufalidis, and Zagier \cite{CGZ} proved one direction of Nahm's conjecture in the following sense:
If the Nahm sum $f_{A,b,c}(q)$ is modular, then the element \eqref{eq:intro xi} associated with the solution of 
Nahm's equation \eqref{eq:intro nahm equation}
in the real interval $(0, 1)^N$ is torsion.
They construct a map
\begin{align*}
  R_\zeta : B(F) / m B(F) \to F_m^\times / F_m^{\times m}
\end{align*}
where $m$ is an integer, $\zeta$ is a primitive $m$th root of unity, and $F_m = F(\zeta)$,
and they show that this map is injective for sufficiently many $m$.
The one direction of Nahm's conjecture follows from this injectivity result
combined with the formula of the asymptotic expansion of Nahm sums given in \cite{GZ-asymptotics}.

In this paper, we study Nahm sums associated with symmetrizable matrices, which have the following form:
\begin{equation}\label{eq:intro nahm symmetrizable}
  f_{A, b, c, d} (q) \coloneqq \sum_{n \in \mathbb{N}^N} 
    \frac{q^{ \frac{1}{2} n^\mathsf{T} A D n + n^\mathsf{T} b + c}}{(q^{d_1}; q^{d_1})_{n_1} \cdots (q^{d_N};q^{d_N})_{n_N}}.
\end{equation}
We have a new input $d = (d_1,\dots,d_N) \in \mathbb{Z}_{>0}^{N}$,
and the symmetric matrix $A$ in the original Nahm sum \eqref{eq:intro nahm} is replaced by a symmetrizable 
matrix $A$ with the symmetrizer $D \coloneqq \diag(d_1,\dots,d_N)$. 
Specifically, in \eqref{eq:intro nahm symmetrizable}, we require that $A D$ is symmetric positive definite rather than $A$ itself.

Prototypical examples of this type of Nahm sums appear in the Kanade-Russell conjecture \cite{KanadeRussell} (in the form by Kur{\c{s}}ung{\"o}z \cite{Kursungoz}),
which gives the mod $9$ version of the Rogers-Ramanujan type identities. 
The conjecture predicts the following $q$-series identities:
\begin{alignat}{2}
  \label{eq:intro KanadeRussell 1}
  f_1 (q) \coloneqq &\sum_{n_1,n_2 \geq 0} \frac{q^{n_1^2 + 3 n_1 n_2 + 3 n_2^2 \phantom{+ 0 n_2 + 0 n_2}}}{(q;q)_{n_1} (q^3;q^3)_{n_2}} &&\overset{?}{=} \frac{1}{(q,q^3,q^6,q^8;q^9)_{\infty}},  \\
  \label{eq:intro KanadeRussell 2}
  f_2 (q) \coloneqq &\sum_{n_1,n_2 \geq 0} \frac{q^{n_1^2 + 3 n_1 n_2 + 3 n_2^2 + n_1 + 3 n_2}}{(q;q)_{n_1} (q^3;q^3)_{n_2}} &&\overset{?}{=} \frac{1}{(q^2,q^3,q^6,q^7;q^9)_{\infty}}, \\
  \label{eq:intro KanadeRussell 3}
  f_3 (q) \coloneqq &\sum_{n_1,n_2 \geq 0} \frac{q^{n_1^2 + 3 n_1 n_2 + 3 n_2^2 + 2 n_1 + 3 n_2}}{(q;q)_{n_1} (q^3;q^3)_{n_2}} &&\overset{?}{=} \frac{1}{(q^3,q^4,q^5,q^6;q^9)_{\infty}}.
\end{alignat}
The Nahm sums in the left-hand sides are given by the matrix 
$A = \begin{bsmallmatrix}
  2 & 1 \\ 3 & 2
\end{bsmallmatrix}$
with the symmetrizer $D = \diag(1,3)$.
The three linear terms are given by $b = \begin{bsmallmatrix} 0 \\ 0 \end{bsmallmatrix}$, $\begin{bsmallmatrix} 1 \\ 3 \end{bsmallmatrix}$, 
and $\begin{bsmallmatrix} 2 \\ 3 \end{bsmallmatrix}$, respectively.
Assuming these identities \eqref{eq:intro KanadeRussell 1}--\eqref{eq:intro KanadeRussell 3}, we see that the vector-valued function
\begin{align*}
  g(\tau) = 
  \begin{pmatrix}
    q^{-1/18} f_1 (q) \\
    q^{5/18} f_2 (q)\\
    q^{11/18} f_3 (q)\\
  \end{pmatrix}
\end{align*}
satisfies the following modular transformation formula:
\begin{align}\label{eq:intro KR modular}
  g(\tau + 1) = 
  \begin{pmatrix}
    \zeta_{18}^{-1} & 0 & 0 \\
    0 & \zeta_{18}^5 & 0 \\ 
    0 & 0 & \zeta_{18}^{11}
  \end{pmatrix}
  g(\tau),\quad
  g(-\frac{1}{\tau}) = 
  \begin{pmatrix}
    \alpha_1 & \alpha_2 & \alpha_4 \\
    \alpha_2 & -\alpha_4 & -\alpha_1 \\ 
    \alpha_4 & -\alpha_1 & \alpha_2
  \end{pmatrix}
  g(\frac{\tau}{3})
\end{align}
where $\alpha_k = \dfrac{1}{2 \sqrt{3} \sin \frac{k \pi}{9}}$.
The formula \eqref{eq:intro KR modular} implies that $g(\tau)$ is a vector-valued modular function on the 
congruence subgroup $\Gamma_0(3) \coloneqq \left\{ \begin{bsmallmatrix} a & b \\  c & d \end{bsmallmatrix} \mid c \equiv 0 \bmod 3 \right\} \subseteq \SL(2, \mathbb{Z})$.

Nahm sums for symmetrizable matrices also appear in various areas of mathematics, such as 
partition identities \cite{KanadeRussell_Staircases,Kursungoz19,Kursungoz,TakigikuTsuchioka}, 
character formulas of principal subspaces for affine Lie algebras of twisted type \cite{ButoracSadowski,HKOTT,okado_parafermionic_2022,PennSadowski17,PennSadowski18},
and $Y$-systems on cluster algebras associated with skew-symmetrizable matrices \cite{mizuno2020difference}.

The purpose of this paper is to demonstrate that the theory of Nahm sums for symmetric matrices 
can be similarly applied to Nahm sums for symmetrizable matrices.
In section \ref{section:nahm conjecture},
we see that one direction of Nahm's conjecture, proved by Calegari, Garoufalidis, and Zagier in \cite{CGZ} for the symmetric case, 
also holds for symmetrizable case (Theorem \ref{thm:torsion of modular}).
As with the symmetric case, the proof uses the injectivity of the map $R_\zeta: B(F) / m B(F) \to F_m^\times / F_m^{\times m}$.
Assuming the injectivity result, the main step of the proof is to study the asymptotic behavior of the Nahm sum at roots of unity, 
which is given in Section \ref{section:asymptotics}.
We provide the explicit formula for the asymptotic expansion of Nahm sums, which is a generalization of the symmetric case given in \cite{GZ-asymptotics}. 
We also need some number-theoretic properties of the constant term in the asymptotic expansion, which are provided in Proposition \ref{prop:in Fm}.
The main ingredient in the constant term is the \emph{cyclic quantum dilogarithm function}.
In Appendix \ref{appendix:cyclic dilog}, we provide some properties of the cyclic quantum dilogarithm function that we use in the proof of Proposition \ref{prop:in Fm}.
Finally, in Section \ref{section:experiments}, we present some experiments on modular Nahm sums for symmetrizable matrices.
We use a numerical method from \cite{Zagier} to search for modular Nahm sum for rank $2$ and rank $3$ cases. 
The resulting lists of modular candidates are given in Table \ref{table:rank2} for rank $2$ case, 
and Table \ref{table:1,1,2} and \ref{table:2,2,1} for rank $3$ case.
For rank $2$ we might expect that Table \ref{table:rank2}, together with symmetric case in \cite[Table 2]{Zagier}, exhausts all modular Nahm sums.
In addition, many candidates in Table \ref{table:rank2} are already known in the literature, and some of them are known to be modular.
For rank $3$, however, the number of modular candidates explodes. We should emphasize that Table \ref{table:1,1,2} and \ref{table:2,2,1} are just the tip of the iceberg.
Providing a complete picture and a systematic understanding are left as a future task.
The only observation on the rank 3 lists in this paper is that there seem to be ``Langlands dual'' pairs of modular Nahm sums.
As an example, we verify numerically that two Nahm sums associated with
$A = \begin{bsmallmatrix}
  1 & 0 & 1/2\\
  0 & 2 & 1\\
  1 & 2 & 2
\end{bsmallmatrix}$ and
$A^{\vee} = \begin{bsmallmatrix}
  1 & 0 & 1\\
  0 & 2 & 2\\
  1/2 & 1 & 2
\end{bsmallmatrix}$, which are related by taking the transpose, are related by the modular $S$-transformation.
More precisely, we will define vector-valued functions $g(\tau)$ and $g^{\vee}(\tau)$
whose components are given by (partial) Nahm sums associated with $A$ and $A^{\vee}$, respectively,
and will conjecture an explicit modular $S$-transformation formula \eqref{eq:modular S for Langlands pair} that 
relate $g(-1/\tau)$ with $g^{\vee}(\tau/2)$ and $g^{\vee}(-1/\tau)$ with $g(\tau/2)$.

\begin{acknowledgment}
  The author would like to thank Shunsuke Tsuchioka for fruitful discussions and for pointing out the identity \eqref{eq:double sum RR x}.
  The author also would like to thank Ole Warnaar for helpful comments.
	This work is supported by JSPS KAKENHI Grant Number JP21J00050.
\end{acknowledgment}

\section{Asymptotics of Nahm sums}
\label{section:asymptotics}
Let $N$ be a natural number.
Suppose that $Q = (A, b, c, d)$ is a quadruple where
$A \in \mathbb{Q}^{N \times N}$ is a matrix, 
$b \in \mathbb{Q}^N$ is a vector,
$c \in \mathbb{Q}$ is a scalar,
and $d \in \mathbb{Z}_{>0}^N$ is a vector,
such that $A D$ is symmetric positive definite where $D = \diag (d_1, \dots, d_N)$.
The \emph{Nahm sum} associated with $Q$ is a $q$-series defined by 
\begin{align}
  f_Q (q) = \sum_{n \in \mathbb{N}^N} \frac{q^{Q(n)}}{(q^{d_1}; q^{d_1})_{n_1} \cdots (q^{d_N};q^{d_N})_{n_N}}
\end{align}
where we denote by, abuse of notation,
\begin{align}
  Q(n) = \frac{1}{2} n^{\mathsf{T}} A D n + n^\mathsf{T} b + c
\end{align}
the quadratic form associated with $Q = (A, b, c, d)$.
We say that $Q$ is \emph{symmetric} if $d_i = 1$ for any $i$. 
We also use the term \emph{symmetrizable} to mean not necessarily symmetric.

\subsection{Asymptotics at roots of unity}
In this section, we study the asymptotics of Nahm sums $f_Q (q)$ at roots of unity.
We think of $f_Q (q)$ as a function on the upper half-plane by setting $q^\lambda = e^{2 \pi i \tau \lambda}$
for any $\lambda \in \mathbb{Q}$ and $\tau \in \mathbb{C}$ with $\Im \tau >0$.
We will consider the limit where $\tau$ tends to a rational number from above in the complex plane.

Let $(z_1, \dots, z_N) \in \mathbb{R}^N$ be the unique solution of the \emph{Nahm's equation}
\begin{align}\label{eq:nahm equation}
  1 - z_i = \prod_{j=1}^N z_j^{A_{ij}}
  \quad \text{for $i = 1, \dots, N$}
\end{align}
such that $0 < z_i < 1$ for any $i$.
In fact, the existence and the uniqueness are proved in the same way as in the symmetric case
(see e.g. \cite[Lemma 2.1]{VlasenkoZwegers})
by using the fact that $A D$ is positive definite.
We define a real number
\begin{align}\label{eq:volume}
  \Lambda = - \sum_{i=1}^N d_i^{-1} \mathrm{L}(z_i)
\end{align}
where 
\begin{align}
  \mathrm{L}(z) = \Li_2 (z) + \frac{1}{2} \log(z) \log(1-z) - \frac{\pi^2}{6}
\end{align}
is the Rogers dilogarithm function shifted by the constant ${\pi^2}/{6}$ so that $\mathrm{L}(1) = 0$.

For a symmetric positive definite matrix $A$ and an analytic function $f(x) = f(x_1, \dots, x_N)$, we consider the formal Gaussian integral:
\begin{align}
  \mathbf{I}_A [f] = \frac{\int_{\mathbb{R}^N} e^{-\frac{1}{2} x^\mathsf{T} A x} f(x)\, dx}
  {\int_{\mathbb{R}^N} e^{-\frac{1}{2} x^\mathsf{T} A x}\, dx}.
\end{align}
For example, we have $\mathbf{I}_A [x_i x_j] = (A^{-1})_{ij}$.
For a positive integer $m$, an $m$th root of unity $\zeta$, and a complex number $w$ with $\lvert w \rvert < 1$,
we define a formal power series 
\begin{align}
  \psi_{w, \zeta} (\nu, \varepsilon) = - \sum_{r\geq2} \sum_{t=1}^{m} \Bigl( B_r \Bigl(1 - \frac{t + \nu}{m} \Bigr) - \delta_{2,r} \frac{\nu^2}{m^2} \Bigr)
    \Li_{2-r} (\zeta^t w) \frac{\varepsilon^{r-1}}{r!},
\end{align}
where $B_r(x)$ is the Bernoulli polynomial of degree $r$ defined by $t e^{tx}/(e^t - 1) = \sum_{r=0}^\infty B_r(x) t^r / r !$,
and $\Li_r (w) = \sum_{k=1}^\infty w^k / k^r$ is the polylogarithm function.
We then consider the formal Gaussian integral
\begin{align}\label{eq:def of I}
  I_{Q, \zeta}(k, \varepsilon) = 
  \mathbf{I}_{\frac{1}{m} \widetilde{A} D} \left[ e^{- \frac{1}{m} x^\mathsf{T} b \varepsilon^{1/2} - \frac{1}{m} (c + \frac{\tr D}{24}) \varepsilon} 
    \prod_{i=1}^{N} \exp \Bigl(\psi_{\zeta_i^{k_i} z_i^{\frac{1}{m}}, \zeta_i} (x_i \varepsilon^{-1/2}, d_i \varepsilon) \Bigr) \right],
\end{align}
for any $k \in (\mathbb{Z}/m \mathbb{Z})^N$ where 
\begin{align}
  \widetilde{A} = A + \diag (z / (1-z))
\end{align}
and $\zeta_i = \zeta^{d_i}$.
The expression \eqref{eq:def of I} has a well-defined meaning as a formal power series in $\mathbb{C} \llbracket \varepsilon \rrbracket$.

For any coprime integers $p$ and $q$, we define the \emph{Dedekind sum} by
\begin{align}
  s (p, q) \coloneqq \sum_{r \bmod q} \dedekind{\frac{r}{q}} \dedekind{\frac{pr}{q}}
\end{align}
where we use the saw wave function defined by
\begin{align}
  ((x)) = 
  \begin{cases}
    x - \floor{x} - \frac{1}{2} &\text{if $x \notin \mathbb{Z}$} \\
    0 &\text{if $x \in \mathbb{Z}$}.
  \end{cases}
\end{align}
For example, we have $s(1, m) = \frac{(m-1)(m-2)}{12m}$.

We say that a positive integer $\delta$ is a \emph{strong denominator} of $Q$ if 
the value of $Q(k)$ modulo $1$ for $k \in \mathbb{Z}^N$ depends only on the residue class of $k$ modulo $\delta$.
We define the quadratic Gauss sum 
\begin{equation}
  G(Q, \alpha) = \frac{1}{\delta^N} \sum_{k \in (\mathbb{Z}/\delta \mathbb{Z})^N} \mathbf{e} (\bar{\alpha} Q(k))
\end{equation}
where $\mathbf{e}(x) = e^{2 \pi i x}$, 
$\bar{\alpha}$ is the reduction of $\alpha$ modulo $\delta$, 
and $\delta$ is a chosen strong denominator of $Q$.
The $G(Q, \alpha)$ does not depend on the choice of $\delta$.

Let $\tilde{f}_Q(\tau) = f_Q(e^{2 \pi i \tau})$.
We now have the following asymptotic formula for the Nahm sum, whose proof is obtained as a minor modification of the proof for the symmetric case in \cite{GZ-asymptotics}.
\begin{theorem}
  Let $Q = (A, b, c, d)$ be a quadruple as above,
  and $\delta$ be a strong denominator of $Q$.
  Suppose that $\alpha$ is a rational number whose denominator $m$ is prime to $\delta$ and $d_1, \dots, d_N$.
  Let $\zeta = \mathbf{e}(\alpha)$ be the corresponding $m$th root of unity.
  Let $\theta_i$ be the $m$th root of $z_i$ such that $\theta_i \in (0, 1)$, where 
  $(z_1, \dots, z_N)$ is the unique solution of the Nahm's equation \eqref{eq:nahm equation}
  such that $z_i \in (0,1)$. Then we have the asymptotic formula
  \begin{align}\label{eq:nahm asymptotic}
    e^{-\frac{\Lambda}{m\varepsilon}} \tilde{f}_Q \Bigl(\alpha + \frac{i\varepsilon}{2 \pi m} \Bigr) \sim
    m^{-\frac{N}{2}} \chi(d, \alpha) c(Q) G(Q, \alpha) S_{Q,\zeta} (\varepsilon)
  \end{align}
  as $\varepsilon$ tends to $0$ from right within the real line.
  Here
  $\chi(d, \alpha) = \prod_{i=1}^N \chi_i$ is the product of 
  the $12$th root of $\zeta$ defined by $\chi_i = \mathbf{e}(s(p_i,q_i)/2)$ where $p_i/q_i = \alpha d_i$,
  \begin{align}
    c(Q) = (\det \widetilde{A} D)^{- \frac{1}{2}} \prod_{i=1}^{N} \theta_i^{b_i/d_i} (1 - z_i)^{\frac{1}{2} - \frac{1}{m}},
  \end{align}
  and
  \begin{align}
    S_{Q, \zeta}(\varepsilon) =
    \biggl(\prod_{i=1}^{N} D_{\zeta_i} (\zeta_i \theta_i)^{- \frac{1}{m}} \biggr)
      \sum_{k \in (\mathbb{Z}/m\mathbb{Z})^N} \zeta^{\overline{Q(k)}}
        \biggl(\prod_{i=1}^{N} \frac{\theta_i^{(k^\mathsf{T}A)_i}}{(\zeta_i \theta_i ; \zeta_i)_{k_i}}\biggr)  I_{Q, \zeta}(k, \varepsilon),
  \end{align}
  where $D_{\zeta} (x)$ is the cyclic quantum dilogarithm function defined by \eqref{eq:cyclic dilog},
  and $\overline{Q(k)}$ is the reduction of $Q(k)$ modulo $m$.
\end{theorem}

\begin{remark}
  We have corrected the following errors in \cite{GZ-asymptotics}:
  \begin{enumerate*}[label=(\alph*)]
    \item the sign of $- \frac{1}{m} x^\mathsf{T} b \varepsilon^{1/2}$ in \eqref{eq:def of I},
    \item the missing term $\frac{\tr D}{24}$ (which reduces $\frac{N}{24}$ for symmetric case) in \eqref{eq:def of I},
    \item the absence of the Dedekind sum in $\chi(d, \alpha)$.
  \end{enumerate*}
\end{remark}

\begin{remark}
  When $\alpha = 0$, the asymptotic formula for symmetrizable Nahm sums was previously obtained in \cite{kanade2019searching}.
\end{remark}

\section{Nahm sums and Bloch groups}
\subsection{Bloch groups}
Let $F$ be a number field.
We denote by $Z(F)$ the (additive) free abelian group on $\mathbb{P}^1(F) = F \sqcup \{ \infty \}$.
We denote by $[X] \in Z(F)$ the element corresponding to $X \in \mathbb{P}^1(F)$.
The \emph{Bloch group} of $F$ is the quotient
\begin{align}
  B(F) = A(F) / C(F)
\end{align}
where $A(F)$ is the kernel of the map
\begin{align}
  d : Z(F) \to \wedge^2 F^\times,\quad
  [X] \mapsto X \wedge (1 - X)
\end{align}
(and $[0], [1], [\infty] \mapsto 0$), and $C(F) \subseteq A(F)$ is the subgroup generated by
the \emph{five-term relation}
\begin{align}\label{eq:five term}
  [X] - [Y] + \left[ \frac{Y}{X} \right] - \left[ \frac{1-X^{-1}}{1-Y^{-1}} \right] + \left[ \frac{1-X}{1-Y} \right]
\end{align}
for any $X, Y \in \mathbb{P}^1 (F)$ except where $\frac{0}{0}$ or $\frac{\infty}{\infty}$ appears in \eqref{eq:five term}.
There are several conventions for the definition of the Bloch group, which agree up to $6$-torsion.
We use the definition in \cite{CGZ}.

Let $F_m$ be a field obtained by adjoining to $F$ a primitive $m$th root of unity $\zeta = \zeta_m$.
The extension $F_m / F$ is Galois, and we have an injective morphism
\begin{align*}
  \chi : \Gal(F_m/F) \to (\mathbb{Z}/m \mathbb{Z})^\times
\end{align*}
determined by $\sigma(\zeta) = \zeta^{\chi (\sigma)}$ for any $\sigma$.
The group $F_m^\times / F_m^{\times m}$ is equipped with the (multiplicative) $\mathbb{Z}/m \mathbb{Z}$-module
structure given by $k \mapsto (x \mapsto x^k)$.
We define the $\chi^{-1}$-eigenspace by
\begin{align*}
  (F_m^\times / F_m^{\times m})^{\chi^{-1}} \coloneqq 
  \{ x \in F_m^\times / F_m^{\times m} \mid \forall \sigma \in \Gal(F_m/F),\ \sigma(x) = x^{\chi(\sigma)^{-1}} \}.
\end{align*}

For a number field $F$ that does not contain any non-trivial $m$th root of unity,
Calegari-Garoufalidis-Zagier \cite{CGZ} defined a map
\begin{align}\label{eq:CGZ map}
  R_\zeta : B(F) / n B(F) \to F_m^\times / F_m^{\times m}
\end{align}
such that its image is contained in the $\chi^{-1}$-eigenspace.
This map is defined by using cyclic dilogarithm function \eqref{eq:cyclic dilog}.
They prove the following injectivity result:
\begin{theorem}{\cite[Theorem 1.2]{CGZ}}\label{theorem:injectivity}
  Let $F$ be a number field. 
  There exists an integer $M$ (depending only on $F$) such that, 
  for any integer $m$ prime to $M$
  and any $m$th root of unity $\zeta$,
  the map $R_\zeta$ is injective.
\end{theorem}
The proof in \cite{CGZ} uses the compatibility between $R_\zeta$ and 
a Chern class map in algebraic $K$-theory, and the injectivity of this Chern class map.

\subsection{Nahm's conjecture for symmetrizable case}
\label{section:nahm conjecture}
We first observe number-theoretic features of the constant term in the asymptotic expansion \eqref{eq:nahm asymptotic}.
Let $Q = (A, b, c, d)$ be a quadruple as in Section \ref{section:asymptotics}.
We fix a strong denominator $\delta$ of $Q$ such that $\delta d_i^{-1} b_i$ and $\delta d_i^{-1} A_{ij}$ are integers for any $i,j$.
Let $F$ be a number field obtained by adjoining $z_1^{1/\delta}, \dots, z_N^{1/\delta}$ to $\mathbb{Q}$,
where $z_i$ is the solution of the Nahm's equation \eqref{eq:nahm equation} such that $0 < z_i < 1$.
Let $m$ be an integer prime to $\delta$ and $d_1, \dots, d_N$.
Recall that $F_m$ is a field obtained by adjoining to $F$ a primitive $m$th root of unity $\zeta$.
We define a number
\begin{align}\label{eq:def of u}
  u =  \biggl( \prod_{i = 1}^{N} \theta_i^{b_i / d_i} (1 - z_i)^{-\frac{1}{m}}
  D_{\zeta_i} (\zeta_i \theta_i)^{-\frac{1}{m}} \biggr)
  \sum_{k \in (\mathbb{Z}/m\mathbb{Z})^N} \zeta^{\overline{Q(k)}}
    \prod_{i=1}^{N} \frac{\theta_i^{(k^\mathsf{T}A)_i}}{(\zeta_i \theta_i ; \zeta_i)_{k_i}}
\end{align}
which is a part of the constant term in the asymptotic formula \eqref{eq:nahm asymptotic}.
We see that the $m$th power of $u$ satisfies the following properties, 
which is stated (without a proof) in \cite[Theorem 7.1]{CGZ} for the symmetric case.
We will provide the proof for the sake of completeness.
\begin{proposition}\label{prop:in Fm}
  We have $u^m \in F_m$.
  Moreover, if $u^m \neq 0$, then its image in $F_m^\times / F_m^{\times m}$
  belongs to the $\chi^{-1}$-eigenspace.
\end{proposition}
\begin{proof}
  We first see that $u$ can be also written as
  \begin{align*}
    u = \biggl( \prod_{i = 1}^{N} \theta_i^{b_i / d_i}
      D_{\zeta_i} (\theta_i)^{-\frac{1}{m}} \biggr)
      a_\zeta (\theta)
  \end{align*}
  where we define
  \begin{equation}\label{eq:in Fm sum part}
    a_\zeta (\theta) = \sum_{k \in (\mathbb{Z}/m\mathbb{Z})^N}  \zeta^{\overline{Q(k)}}
    \prod_{i=1}^{N} \frac{\theta_i^{(k^\mathsf{T}A)_i}}{(\theta_i ; \zeta_i)_{k_i + 1}}.
  \end{equation}
  This follows from the formula
  \begin{align}\label{eq:D zeta x / D x}
    \frac{D_\zeta (\zeta x)}{D_\zeta (x)} = \frac{(1-x)^m}{1-x^m}.
  \end{align}
  Also note that the summand in \eqref{eq:in Fm sum part} is well-defined over the set $(\mathbb{Z}/m \mathbb{Z})^N$,
  which follows from the fact that $\theta_i^m$ is a solution of the Nahm's equation \eqref{eq:nahm equation} (see Lemma 3.1 in \cite{DG18}). 

  Let $H$ be the Kummer extension of $F_m$ obtained by adjoining the $m$th roots 
  of $z_1^{1/\delta}, \dots, z_N^{1/\delta}$ to $F_m$.
  We see that $u^m \in H$ by definition. We now see that $u^m$ belongs to the smaller field $F_m$ by showing that it is fixed by any automorphism $\tau \in \Gal(H / F_m)$. 
  Let $e_i$ be integers such that $\zeta_i^{e_i} = \tau(\theta_i)/\theta_i$. Then we see that
  \begin{align}\label{eq:u^m in Fm 1}
    \tau \biggl( \prod_{i = 1}^{N} 
    D_{\zeta_i} (\theta_i)^{-1} \biggr) = 
    \prod_{i = 1}^{N} 
    D_{\zeta_i} (\theta_i)^{-1} \frac{\theta_i^{m(e^\mathsf{T} A)_i}}{(\theta_i; \zeta_i)_{e_i}^m}
  \end{align}
  by repeatedly using \eqref{eq:D zeta x / D x} and also using the fact that $\theta_i^m$ is a solution of the Nahm's equation.
  We also see that
  \begin{align}\label{eq:u^m in Fm 2}
    \biggl(\prod_{i=1}^{N} \frac{\theta_i^{m(e^\mathsf{T} A)_i}}{(\theta_i; \zeta_i)_{e_i}^m} \biggr) \cdot \tau(a_\zeta(\theta)^m) = 
    \biggl(\sum_{k \in (\mathbb{Z}/m\mathbb{Z})^N}  \zeta^{\overline{Q(k)}}
    \prod_{i=1}^{N} \frac{\theta_i^{((k+e)^\mathsf{T} A)_i}}{(\theta_i ; \zeta_i)_{k_i + e_i + 1}} \biggr)^m 
    = a_\zeta(\theta)^m.
  \end{align}
  The equations \eqref{eq:u^m in Fm 1} and \eqref{eq:u^m in Fm 2} imply that $\tau(u^m) = u^m$.

  Now we assume that $u^m \neq 0$.
  Suppose that $\sigma \in \Gal(F_m / F)$.
  Let $q = \chi(\sigma) \in (\mathbb{Z}/m \mathbb{Z})^\times$, and let $p \in \mathbb{Z}_{>0}$ be a lift of $q^{-1}$.
  Let $\tilde{\sigma} \in \Gal(H / F)$ be a lift of $\sigma$,
  and define $\tilde{\theta}_i \coloneqq \tilde{\sigma}(\theta_i)$.
  Let $s_i$ be integers such that $\zeta_i^{s_i} = \tilde{\theta}_i / \theta_i$.
  We define an element
  \begin{equation}\label{eq:in Fm def of v}
    v =
    \biggl(\prod_{i = 1}^{N} h_i \biggr) \cdot
    \frac{a_{\zeta^q} (\tilde{\theta})}{a_\zeta(\theta)^p}  \quad \in H^\times
  \end{equation}
  where
  \begin{equation}\label{eq:in Fm def of h}
    h_i = \frac{\tilde{\theta}_i^{b_i / d_i}} {\theta_i^{b_i p / d_i}} \cdot 
    \frac{\theta_i^{p (s^\mathsf{T} A)_i}}{(\theta_i;\zeta_i^q)_{ps_i}} \cdot
    \prod_{t=1}^{p-1} \frac{1}{(\theta_i; \zeta_i)_{qt}} \quad \in H^\times.
  \end{equation}
  Then we have $v^m \equiv \sigma(u^m) / u^{mp} \bmod F_m^{\times m}$ by Proposition \ref{prop:cyclic dilog change of zeta}.
  Thus, to prove the proposition, it suffices to see that $v \in F_m$ by showing that $v$ is fixed by any automorphism $\tau \in \Gal(H/F_m)$.
  Again, let $e_i$ be integers such that $\zeta_i^{e_i} = \tau(\theta_i)/\theta_i$.
  We see that
  \begin{equation}\label{eq:in Fm def tau h}
    \begin{split}
          \prod_{i=1}^{N} \frac{\tau(h_i)}{h_i} &= 
          \zeta^{ p s^\mathsf{T} AD e + b e^\mathsf{T} - p b e^\mathsf{T}} 
          \prod_{i=1}^N 
            \frac{(\theta_i; \zeta_i)_{e_i}^p (\theta_i; \zeta_i^q)_{ps_i}}{(\theta_i; \zeta_i^q)_{pe_i} (\zeta^{e_i} \theta_i; \zeta_i^q)_{ps_i}} \\
            &= \zeta^{p s^\mathsf{T} AD e + b e^\mathsf{T} - p b e^\mathsf{T}}
            \prod_{i=1}^N 
            \frac{(\theta_i; \zeta_i)_{e_i}^p}
            {(\tilde{\theta}_i; \zeta_i^q)_{pe_i}} \\
            &= \zeta^{b e^\mathsf{T} - p b e^\mathsf{T}}
            \prod_{i=1}^N 
            \frac{\tilde{\theta}_i^{p(e^\mathsf{T} A)_i} (\theta_i; \zeta_i)_{e_i}^p}
            {\theta_i^{p(e^\mathsf{T} A)_i} (\tilde{\theta}_i; \zeta_i^q)_{pe_i}},
    \end{split}
  \end{equation}
  where we use \eqref{eq:cyclic dilog change of zeta 3} in the first equality.
  We now compute
  \begin{equation}\label{eq:in Fm def tau tilde a}
      \begin{split}
        \zeta^{b e^\mathsf{T}} 
          \biggl(\prod_{i=1}^N \frac{\tilde{\theta}_i^{p (e^\mathsf{T} A)_i}}{(\tilde{\theta}_i; \zeta_i^q)_{pe_i}} \biggr)
          \tau(a_{\zeta^q}(\tilde{\theta})) 
        &=
        \sum_{k \in (\mathbb{Z}/m\mathbb{Z})^N} \zeta^{q \overline{Q(pk)} + p k^\mathsf{T} A D e + b e^\mathsf{T}}
        \prod_{i=1}^{N} \frac{\tilde{\theta}_i^{p((k+e)^\mathsf{T} A)_i}}{(\tilde{\theta}_i ; \zeta_i^q)_{p(k_i+e_i) + 1}} \\
        &=
        \sum_{k \in (\mathbb{Z}/m\mathbb{Z})^N} \zeta^{q \overline{Q(p(k-e))} + p (k-e)^\mathsf{T} A D e + b e^\mathsf{T}}
        \prod_{i=1}^{N} \frac{\tilde{\theta}_i^{p(k^\mathsf{T}A)_i}}{(\tilde{\theta}_i ; \zeta_i^q)_{pk_i + 1}} \\
        &= \zeta^{- p e^\mathsf{T} AD e} a_{\zeta^q}(\tilde{\theta}).
      \end{split}
  \end{equation}
  Similarly, we also compute
  \begin{equation}\label{eq:in Fm def tau a}
    \begin{split}
      \zeta^{p b e^\mathsf{T}} 
        \biggl(\prod_{i=1}^N \frac{\theta_i^{p (e^\mathsf{T} A)_i}}{(\theta_i; \zeta_i)_{e_i}^p} \biggr)
        \tau(a_{\zeta}(\theta))^p
      &=
      \biggl(\sum_{k \in (\mathbb{Z}/m\mathbb{Z})^N} \zeta^{\overline{Q(k)} + k^\mathsf{T} A D e + b e^\mathsf{T}}
      \prod_{i=1}^{N} \frac{\theta_i^{((k+e)^\mathsf{T} A)_i}}{(\theta_i ; \zeta_i)_{k_i+e_i + 1}} \biggr)^p \\
      &=
      \biggl(\sum_{k \in (\mathbb{Z}/m\mathbb{Z})^N} \zeta^{\overline{Q(k-e)} + (k-e)^\mathsf{T} A D e + b e^\mathsf{T}}
      \prod_{i=1}^{N} \frac{\theta_i^{(k^\mathsf{T} A)_i}}{(\theta_i ; \zeta_i)_{k_i + 1}} \biggr)^p \\
      &= \zeta^{- p e^\mathsf{T} AD e} a_\zeta(\theta)^p.
    \end{split}
  \end{equation}
  Combining \eqref{eq:in Fm def of v}--\eqref{eq:in Fm def tau a},
  we see that $\tau(v) = v$, which completes the proof.
  \end{proof}

We now prove the one direction of Nahm's conjecture for symmetrizable case.
We define an element in the Bloch group $B(F)$ by
\begin{align*}
  \xi_{A,d} = \sum_{i=1}^N d_i^{-1} [z_i].
\end{align*}
To be precise, $\xi_{A,d}$ itself does not belong to $B(F)$, 
rather a multiple of $\xi_{A,d}$ by at least $\lcm(d_1, \dots, d_N)$ belongs to $B(F)$.
But we shall not concern ourselves too much with that since we will only care whether $\xi_{A,d}$ is torsion or not.
Finally, we say that the Nahm sum $f_{A,b,c,d}(q)$ is modular to mean that the function $\tilde{f}(\tau) = f_{A,b,c,d}(e^{2 \pi i \tau})$ is 
invariant with respect to some finite index subgroup of $\SL(2, \mathbb{Z})$.

\begin{theorem}\label{thm:torsion of modular}
  Suppose that the Nahm sum $f_Q(q)$ associated with $Q = (A, b, c, d)$ is modular.
  Then $\xi_{A, d}$ is torsion in the Bloch group $B(F)$.
\end{theorem}
\begin{proof}
  Suppose that $\tilde{f}(\tau) = f_{A,b,c,d}(e^{2 \pi i \tau})$ is modular with respect to a finite index subgroup $\Gamma \subseteq \SL(2,\mathbb{Z})$.
  We can assume that $\Gamma$ is contained in the principal congruence subgroup $\Gamma(M)$ for a fixed integer $M$
  by replacing $\Gamma$ by its intersection with $\Gamma(M)$.

  By the asymptotic formula \eqref{eq:nahm asymptotic} for $\alpha = 0$, we have
  \begin{align}\label{eq:nahm asymptotic alpha 0}
    \tilde{f} \Bigl(\frac{i \varepsilon}{2 \pi}\Bigr) = e^{-\frac{\Lambda}{\varepsilon}} (K + O(\varepsilon))  
  \end{align}
  as $\varepsilon$ tends to $0$ from the right, where $K$ is an algebraic number such that $K^2 \in F^\times$.
  The number $\lambda = \Lambda / (2 \pi)^2$ must be rational since the modularity of $\tilde{f}(\tau)$ implies that 
  the function $\tilde{f} (-1/\tau)$ is invariant under some power of $\begin{bsmallmatrix}
    1 & 1 \\ 0 & 1
  \end{bsmallmatrix}$.
  We denote by $\ell$ the denominator of $\lambda$.

  For any positive real number $h$ and any $\begin{bsmallmatrix}
    p & q \\ r & s
  \end{bsmallmatrix} \in \Gamma$, taking $\varepsilon = \frac{s h}{1 - irh/2\pi}$, we find that
  \begin{align}\label{eq:proof of torsion of modular}
    \tilde{f}\Bigl( \frac{i \varepsilon}{2 \pi} \Bigr)
    = \tilde{f}\Bigl( \frac{p i \varepsilon/2\pi + q}{ri\varepsilon/2\pi + s} \Bigr) 
    =\tilde{f}\Bigl(\alpha + \frac{i h}{2 \pi m} \Bigr)
  \end{align}
  where we set $\alpha = q/s$ and $m = s$.
  Comparing the asymptotics of the left and the right expressions in \eqref{eq:proof of torsion of modular} 
  for $h$ tending to $0$ on the positive real line by using \eqref{eq:nahm asymptotic alpha 0} and \eqref{eq:nahm asymptotic},
  we have
  \begin{align}\label{eq:proof of torsion of modular 2}
    K \mathbf{e}(\lambda r / m) = \mu u_\zeta
  \end{align}
  where $u_\zeta = u$ is defined by \eqref{eq:def of u}, and $\mu$ is an algebraic number such that $\mu^{12} \in F_m^{\times}$.
  The equation \eqref{eq:proof of torsion of modular 2} implies that $u_\zeta^{12\ell} \in F_m^{\times}$.
  Since $\Gamma$ is contained in $\Gamma(M)$, we see that $m$ is prime to $M$.

  Consequently, we see that there are infinitely many $m$ such that $u_\zeta^{12\ell}$ for a primitive $m$th root of unity $\zeta$ belongs to $F_m^{\times}$.
  By Proposition \ref{prop:in Fm} and \cite[Lemma 2.4 (e) and Remark 2.6]{CGZ},
  this implies that $R_\zeta(\xi_{A,d})^{12\ell}$ is trivial in $F_m^\times / F_m^{\times m}$ for infinitely many $m$.
  Thus, by the injectivity result Theorem \ref{theorem:injectivity}, we see that $12 \ell \xi_{A,d}$ has a trivial image in $B(F)/m B(F)$ for infinitely many $m$.
  If $\xi_{A,d}$ is not torsion, this contradicts the fact that $B(F)$ is a finitely generated abelian group.
\end{proof}

\section{Experiments on modular Nahm sums}
\label{section:experiments}
In this section, we present some experiments on modular Nahm sums for symmetrizable matrices
by using the numerical method explained in \cite{Zagier}.
The method aims to detect, for a given $q$-series $f(q)$, whether $q^c f(q)$ is likely modular for some $c$.
The actual procedure is as follows:
\begin{enumerate}
  \item Compute $\phi(N) = N \log f(e^{-1/N})$ for four successive values (say $N = 20, 21, 22, 23$).
  \item Take the third difference of the values computed in (1). If this value is extremely small, 
  then the function $q^{c} f(q)$ is likely modular for some $c$. 
\end{enumerate}
This method is justified by the following reasoning:
If $q^{c} f(q)$ is modular for some $c$, $\phi(N)$ is approximated to high precision by a quadratic polynomial in $n$ (see \cite[Lemma 3.1]{VlasenkoZwegers}),
and thus its third difference should be extremely small.
For Nahm sums we actually only need three values in the step (1)
since we know the constant term of the quadratic polynomial by the formula \eqref{eq:volume}.

We use this method to give lists of candidates of modular Nahm sums in rank $2$ (Table \ref{table:rank2}), and 
rank $3$ with specific symmetrizers (Table \ref{table:1,1,2} for $d = (1,1,2)$ and Table \ref{table:2,2,1} for $d = (2,2,1)$).
For rank $2$, 
we present all candidates found in the search 
\begin{equation*}
  A = \frac{1}{a'}
\begin{bmatrix} 
  a_{11} & a_{12}\\  
  d_1^{-1} d_2 a_{12} & a_{22}
\end{bmatrix},\quad 
b = \frac{1}{b'} \begin{bmatrix} b_1 \\ b_2 \end{bmatrix}
\end{equation*}
for $-10 \leq a_{ij}, b_i \leq 20$ and $1 \leq a', b', d_i \leq 10$.
We might expect that these exhaust all modular Nahm sums.
For rank $3$, we find many more candidates.
In Table \ref{table:1,1,2} and Table \ref{table:2,2,1}, we only present the candidates with relatively small numerators and denominators in $A$ and $b$,
and actually there are many other candidates that we have not included here due to space constraints.
The program used for the search is available here\footnote{\url{https://github.com/yuma-mizuno/modular-search-for-Nahm-sums-with-symmetrizers}}.

\begin{table}[t]
  \begin{adjustbox}{valign=t}
    \begin{tblr}{
        colspec = {Q[l,m,mode=math]Q[c,m,mode=math]Q[r,m,mode=math]Q[r,m,mode=math]},
        row{3,4,14,15,16,17,18} = {rowsep+=2pt},
      } 
      \hline
      A & b & c & d \\ \hline
      \SetCell[r=5]{m}{\begin{pmatrix}
        2 & 1 \\
        2 & 2
      \end{pmatrix}}
      & \begin{bsmallmatrix} 0 \\ 0 \end{bsmallmatrix} & -\frac{1}{24} 
      & \SetCell[r=5]{m}{(1,2)} \\
      & \begin{bsmallmatrix} 0 \\ 1 \end{bsmallmatrix} & \frac{1}{24} \\ 
      & \begin{bsmallmatrix} -1 \\ -1 \end{bsmallmatrix} & -\frac{1}{24} \\ 
      & \begin{bsmallmatrix} -1/2 \\ 0 \end{bsmallmatrix} & -\frac{1}{48} \\ 
      & \begin{bsmallmatrix} 1 \\ 2 \end{bsmallmatrix} & \frac{7}{24} & \\ 
      \hline
      \SetCell[r=3]{m}{\begin{pmatrix}
        1 & 1/2 \\
        1 & 1
      \end{pmatrix}}
      & \begin{bsmallmatrix} 0 \\ 0 \end{bsmallmatrix} & -\frac{3}{56}
      & \SetCell[r=3]{m}{(1,2)} \\
      & \begin{bsmallmatrix} 0 \\ 1 \end{bsmallmatrix} & \frac{1}{56} \\ 
      & \begin{bsmallmatrix} 1 \\ 1 \end{bsmallmatrix} & \frac{9}{56} \\ 
      \hline
      \SetCell[r=1]{m}{\begin{pmatrix}
        4 & 2 \\
        6 & 4
      \end{pmatrix}} 
      & \begin{bsmallmatrix} 0 \\ 0 \end{bsmallmatrix} & -\frac{1}{24}
      & \SetCell[r=1]{m}{(1,3)} \\
      \hline
      \SetCell[r=3]{m}{\begin{pmatrix}
        2 & 1 \\
        3 & 2
      \end{pmatrix}}
      & \begin{bsmallmatrix} 0 \\ 0 \end{bsmallmatrix} & -\frac{1}{18}
      & \SetCell[r=3]{m}{(1,3)} \\
      & \begin{bsmallmatrix} 1 \\ 3 \end{bsmallmatrix} & \frac{5}{18}   \\ 
      & \begin{bsmallmatrix} 2 \\ 3 \end{bsmallmatrix} & \frac{11}{18} \\ 
      \hline
      \SetCell[r=2]{m}{\begin{pmatrix}
          3 & 2 \\
          4 & 4
        \end{pmatrix}}
      & \begin{bsmallmatrix} -1/2 \\ 0 \end{bsmallmatrix} & -\frac{1}{60}
      & \SetCell[r=2]{m}{(1,2)} \\
      & \begin{bsmallmatrix} 1/2 \\ 2 \end{bsmallmatrix} & \frac{11}{60}   \\
      \hline
      \SetCell[r=3]{m}{\begin{pmatrix}
          3/2 & 1/2 \\
          1   & 1
        \end{pmatrix}}
      & \begin{bsmallmatrix} -1 \\ 1  \end{bsmallmatrix} & \frac{9}{40} 
      & \SetCell[r=3]{m}{(1,2)} \\
      & \begin{bsmallmatrix} -1/2 \\ 0  \end{bsmallmatrix} & \frac{-1}{40}    \\
      & \begin{bsmallmatrix} 0 \\ 1   \end{bsmallmatrix} & \frac{1}{40}     \\
      \hline
      \SetCell[r=2]{m}{\begin{pmatrix}
        3 & 1 \\
        4 & 2
      \end{pmatrix}}
    & \begin{bsmallmatrix} -1/2 \\ 0 \end{bsmallmatrix} & -\frac{1}{16}
    & \SetCell[r=2]{m}{(1,4)} \\
    & \begin{bsmallmatrix} 3/2 \\ 4 \end{bsmallmatrix} & \frac{7}{16}   \\
    \hline
    \end{tblr} 
    \end{adjustbox}\hfill
    \begin{adjustbox}{valign=t}
    \begin{tblr}{
      colspec = {Q[l,m,mode=math]Q[c,m,mode=math]Q[r,m,mode=math]Q[r,m,mode=math]},
      } 
      \hline
      A & b & c & d\\ \hline 
      \SetCell[r=5]{l,m} \begin{pmatrix}
        1 & -1/2 \\
        -1 & 1
      \end{pmatrix} 
      & \begin{bsmallmatrix} 0 \\ 0 \end{bsmallmatrix} & -\frac{1}{12} 
      & \SetCell[r=5]{m}{(1,2)} \\
      & \begin{bsmallmatrix} -1/2 \\ 1 \end{bsmallmatrix} & \frac{1}{12} \\ 
      & \begin{bsmallmatrix} -1/2 \\ 0 \end{bsmallmatrix} & \frac{1}{12} \\ 
      & \begin{bsmallmatrix} -1/2 \\ 1/2 \end{bsmallmatrix} & \frac{1}{48} \\ 
      & \begin{bsmallmatrix} 0 \\ 1 \end{bsmallmatrix} & \frac{1}{12} \\ 
      \hline
      \SetCell[r=3]{m} \begin{pmatrix}
        2 & -1 \\
        -2 & 2
      \end{pmatrix} 
      & \begin{bsmallmatrix} 0 \\ 0 \end{bsmallmatrix} & -\frac{1}{14}
      & \SetCell[r=3]{m}{(1,2)} \\
      & \begin{bsmallmatrix} -1 \\ 2 \end{bsmallmatrix} & \frac{5}{14} \\ 
      & \begin{bsmallmatrix} 1 \\ 0 \end{bsmallmatrix} & \frac{3}{14} \\ 
      \hline
      \SetCell[r=1]{m} \begin{pmatrix}
        1 & -1/2 \\
        -3/2 & 1
      \end{pmatrix} 
      & \begin{bsmallmatrix} 0 \\ 0 \end{bsmallmatrix} & -\frac{1}{8} 
      & \SetCell[r=1]{m}{(1,3)} \\
      \hline
      \SetCell[r=3]{m} \begin{pmatrix}
        2 & -1 \\
        -3 & 2
      \end{pmatrix} 
      & \begin{bsmallmatrix} 0 \\ 0 \end{bsmallmatrix} & -\frac{1}{9}       
      & \SetCell[r=3]{m}{(1,3)} \\
      & \begin{bsmallmatrix} -1 \\ 3 \end{bsmallmatrix} & \frac{5}{9} \\ 
      & \begin{bsmallmatrix} 1 \\ 0 \end{bsmallmatrix} & \frac{2}{9} \\ 
      \hline
      \SetCell[r=2]{m} \begin{pmatrix}
        1 & -1/2 \\
        -1 & 3/4
      \end{pmatrix} 
      & \begin{bsmallmatrix} -1/2 \\ 1/2 \end{bsmallmatrix} & \frac{1}{60}      
      & \SetCell[r=2]{m}{(1,2)} \\
      & \begin{bsmallmatrix} -1/2 \\ 1 \end{bsmallmatrix} & \frac{1}{15} \\ 
      \hline
      \SetCell[r=3]{m} \begin{pmatrix}
        1 & -1/2 \\
        -1 & 3/2
      \end{pmatrix} 
      & \begin{bsmallmatrix} -3/2 \\ 5/2 \end{bsmallmatrix} & \frac{41}{40}
      & \SetCell[r=3]{m}{(1,2)} \\
      & \begin{bsmallmatrix} -1/2 \\ 1/2 \end{bsmallmatrix} & \frac{1}{40} \\ 
      & \begin{bsmallmatrix} -1/2 \\ 3/2 \end{bsmallmatrix} & \frac{9}{40} \\ 
      \hline
      \SetCell[r=2]{m} \begin{pmatrix}
        1 & -1/2 \\
        -2 & 3/2
      \end{pmatrix} 
      & \begin{bsmallmatrix} -1/2 \\ 1 \end{bsmallmatrix} & -\frac{1}{48}
      & \SetCell[r=2]{m}{(1,4)} \\
      & \begin{bsmallmatrix} -1/2 \\ 3 \end{bsmallmatrix} & \frac{23}{48} \\ 
      \hline
    \end{tblr}
  \end{adjustbox}
  \caption{Symmetrizable modular candidates for rank $2$.}
  \label{table:rank2}
\end{table}

We provide some comments on the lists.
The matrix $A$ in the right table in Table \ref{table:rank2} is the inverse of the matrix $A$ 
of the same row in the left table. More precisely, we expect that the following holds in general:
\begin{conjecture}
  Suppose that the Nahm sum associated with $(A, b, c, d)$ is modular.
  Then the Nahm sum associated with $(A^*, b^*, c^*, d^*)$ is also modular, where
  \begin{align*}
    A^* = A^{-1},\quad
    b^* = A^{-1} b, \quad
    c^* = \frac{1}{2} b^{\mathsf{T}} (AD)^{-1} b - \frac{\tr D}{24} - c, \quad
    d^* = d.
  \end{align*}
\end{conjecture}

Next, we provide case-by-case comments on several Nahm sums in Table \ref{table:rank2}.
\begin{itemize}[wide=0pt]
  \item $A = \begin{bsmallmatrix}
    2 & 1 \\
    2 & 2
  \end{bsmallmatrix}$.
  For $b = \begin{bsmallmatrix}
    0 \\0
  \end{bsmallmatrix}$,
  $\begin{bsmallmatrix}
    0 \\ 1
  \end{bsmallmatrix}$, and
  $\begin{bsmallmatrix}
    1 \\ 2
  \end{bsmallmatrix}$, the Nahm sum is the rank $2$ case of the sum-side of 
  the Bressoud identity \cite{Bressoud80}.
  This fact implies the modularity
  since the product-side of the identity can be written as a ratio of theta functions by 
  the Jacobi's triple product identity.

  \item $A = \begin{bsmallmatrix}
    1 & 1/2 \\
    1 & 1
  \end{bsmallmatrix}$.
The Nahm sum for $\begin{bsmallmatrix} 0 \\ 0 \end{bsmallmatrix}$ is the rank $2$ case of the sum-side of the identity proved by Warnaar \cite[(5.14)]{Warnaar03}, which implies the modularity.
To provide the explicit modular transformation formula, for $\sigma = 0,1$, we define
\begin{align}
  \label{eq:B2 inv -3/56}
  f_{-3/56, \sigma} (q) &\coloneqq q^{-\frac{3}{56}} \sum_{\substack{n_1,n_2 \geq 0 \\ n_1 \equiv \sigma \bmod 2}} \frac{q^{\frac{1}{2} n_1^2 + n_1 n_2 + n_2^2 }}{(q;q)_{n_1} (q^2;q^2)_{n_2}} \\
  \label{eq:B2 inv 1/56}
  f_{1/56, \sigma} (q) &\coloneqq q^{\frac{1}{56}} \sum_{\substack{n_1,n_2 \geq 0 \\ n_1 \equiv \sigma \bmod 2}} \frac{q^{\frac{1}{2} n_1^2 + n_1 n_2 + n_2^2 + n_2}}{(q;q)_{n_1} (q^2;q^2)_{n_2}}  \\
  \label{eq:B2 inv 9/56}
  f_{9/56, \sigma} (q) &\coloneqq q^{\frac{9}{56}} \sum_{\substack{n_1,n_2 \geq 0 \\ n_1 \equiv \sigma \bmod 2}} \frac{q^{\frac{1}{2} n_1^2 + n_1 n_2 + n_2^2 + n_1 + n_2}}{(q;q)_{n_1} (q^2;q^2)_{n_2}} 
\end{align}
so that $f_{c, 0}(q) + f_{c, 1} (q)$ is the Nahm sum associated with $A$.
We conjecture that the following modular transformation formula holds:
\begin{align}
  g(-\frac{1}{\tau}) = 
  \begin{pmatrix}
    S & S \\
    S & -S
  \end{pmatrix}
  g(\frac{\tau}{2}),
\end{align}
where $g(\tau) = (f_{-3/56, 0} (q), f_{1/56, 1} (q), f_{9/56, 1} (q), f_{-3/56, 1} (q), f_{1/56, 0} (q), f_{9/56, 0} (q) )^\mathsf{T}$ and
\begin{align}
  S = 
  \begin{pmatrix}
    \alpha_3 & \alpha_2 & \alpha_1 \\ 
    \alpha_2 & -\alpha_1 & -\alpha_3 \\
    \alpha_1 & -\alpha_3 & \alpha_2
  \end{pmatrix},
  \quad 
  \alpha_k = \sqrt{\frac{2}{7}} \sin \frac{k \pi}{7}
\end{align}
This presumably follows from the Warnaar's identity, together with similar possible identities involving linear terms, but we will not explore it in detail here.

\item $A = \begin{bsmallmatrix} 2 & 1 \\ 3 & 2 \end{bsmallmatrix}$.
The modularity follows from the Kanade-Russell conjecture \eqref{eq:intro KanadeRussell 1}--\eqref{eq:intro KanadeRussell 3}, as we mentioned in the introduction.

\item $A = \begin{bsmallmatrix} 3 & 2 \\ 4 & 4 \end{bsmallmatrix}$.
We have the following identities:
\begin{alignat}{3}
  \label{eq:double sum RR 1}
  &\sum_{n_1, n_2 \geq 0}
  \frac{q^{\frac{3}{2} n_1^2 + 4 n_1n_2 + 4 n_2^2 - \frac{1}{2} n_1 \phantom{+ 0n_2}}}{(q;q)_{n_1} (q^2;q^2)_{n_2}}
  &&= \sum_{n \geq 0} \frac{q^{n^2}}{(q;q)_n} 
  &&= \frac{1}{(q^2,q^3;q^5)_\infty}, \\
  \label{eq:double sum RR 2}
  &\sum_{n_1, n_2 \geq 0}
  \frac{q^{\frac{3}{2} n_1^2 + 4 n_1n_2 + 4 n_2^2 + \frac{1}{2} n_1 + 2 n_2}}{(q;q)_{n_1} (q^2;q^2)_{n_2}}
  &&= \sum_{n \geq 0} \frac{q^{n^2 + n}}{(q;q)_n} 
  &&= \frac{1}{(q^2,q^3;q^5)_\infty}.
\end{alignat}
The second equalities in \eqref{eq:double sum RR 1} and \eqref{eq:double sum RR 2} are the Rogers-Ramanujan identities.
For the first equalities, the author learned the following proof from Shunsuke Tsuchioka.
We can verify that the identities
\begin{align}\label{eq:double sum RR x}
  \sum_{n_1, n_2 \geq 0}
  \frac{q^{\frac{3}{2} n_1^2 + 4 n_1n_2 + 4 n_2^2 - \frac{1}{2} n_1}}{(q;q)_{n_1} (q^2;q^2)_{n_2}} x^{n_1 + 2 n_2}
  = \sum_{n \geq 0} \frac{q^{n^2}}{(q;q)_n} x^n
\end{align}
holds by showing that both sides satisfy the same $q$-difference equation,
since the coefficients of $x^0$ (resp. $x^n$ for $n < 0$) of the both sides are $1$ (resp. $0$).
The first equalities in \eqref{eq:double sum RR 1} and \eqref{eq:double sum RR 2} are obtained by substituting $x$ by $1$ and $q$, respectively, 
into \eqref{eq:double sum RR x}.
In fact, the desired $q$-difference equation can be found algorithmically by using the $q$-version of 
Sister Celine's technique (e.g., \texttt{qMultiSum} package \cite{qMultiSum} can handle such computations).

\item $A = \begin{bsmallmatrix} 3/2 & 1/2 \\ 1 & 1 \end{bsmallmatrix}$.
We have the following identities:
\begin{alignat}{4}
  \label{eq:double sum RR mod 20 1}
  &\sum_{\substack{n_1, n_2 \geq 0 \\ n_1 = 0 \bmod 2}}
  \frac{q^{\frac{3}{4} n_1^2 + n_1n_2 + n_2^2 - n_1 + n_2}}{(q;q)_{n_1} (q^2;q^2)_{n_2}} \notag\\
  &=\sum_{\substack{n_1, n_2 \geq 0 \\ n_1 = 1 \bmod 2}}
  \frac{q^{\frac{3}{4} n_1^2 + n_1n_2 + n_2^2 - \frac{1}{2} n_1}}{(q;q)_{n_1} (q^2;q^2)_{n_2}}
  &&= \sum_{n \geq 0} \frac{q^{n^2+n}}{(q;q)_{2n+1}}
  &&= \frac{(q,q^9, q^{10}; q^{10})_\infty (q^8,q^{12}, q^{20})}{(q;q)_\infty}, \\
  \label{eq:double sum RR mod 20 2}
  &\sum_{\substack{n_1, n_2 \geq 0 \\ n_1 = 1 \bmod 2}}
  \frac{q^{\frac{3}{4} n_1^2 + n_1n_2 + n_2^2 - n_1 + n_2}}{(q;q)_{n_1} (q^2;q^2)_{n_2}} \notag\\
  &=\sum_{\substack{n_1, n_2 \geq 0 \\ n_1 = 0 \bmod 2}}
  \frac{q^{\frac{3}{4} n_1^2 + n_1n_2 + n_2^2 - \frac{1}{2} n_1}}{(q;q)_{n_1} (q^2;q^2)_{n_2}}
  &&= \sum_{n \geq 0} \frac{q^{n^2}}{(q;q)_{2n}} 
  &&= \frac{(q^2,q^8, q^{10}; q^{10})_\infty (q^6,q^{14}, q^{20})}{(q;q)_\infty}, \\
  \label{eq:double sum RR mod 20 3}
  &\sum_{n_1, n_2 \geq 0}
  \frac{q^{\frac{3}{4} n_1^2 + n_1n_2 + n_2^2 + n_2}}{(q;q)_{n_1} (q^2;q^2)_{n_2}}
  &&= \sum_{n \geq 0} \frac{q^{n^2 + n}}{(q;q)_{2n}} 
  &&= \frac{(q^3,q^7, q^{10}; q^{10})_\infty (q^4,q^{16}, q^{20})}{(q;q)_\infty}, \\
  \label{eq:double sum RR mod 20 4}
  &\sum_{n_1, n_2 \geq 1}
  \frac{q^{\frac{3}{4} n_1^2 + n_1n_2 + n_2^2 + n_2}}{(q;q)_{n_1} (q^2;q^2)_{n_2}}
  &&= \sum_{n \geq 0} \frac{q^{n^2 + 2n}}{(q;q)_{2n+1}} 
  &&= \frac{(q^4,q^6, q^{10}; q^{10})_\infty (q^2,q^{18}, q^{20})}{(q;q)_\infty}.
\end{alignat}
The rightmost equalities in 
\eqref{eq:double sum RR mod 20 1}, \eqref{eq:double sum RR mod 20 2}, \eqref{eq:double sum RR mod 20 3}, and \eqref{eq:double sum RR mod 20 4} are
(99), (98), (94), and (96), respectively, in the Slater's list \cite{Slater52}.
Other equalities can be easily verified in the same way as for the left equalities in \eqref{eq:double sum RR 1} and \eqref{eq:double sum RR 2} explained earlier.

\item $A = \begin{bsmallmatrix} 4 & 2 \\ 6 & 4 \end{bsmallmatrix}$.
The modularity follows from the Capparelli identities \cite{Capparelli} in the form presented in \cite{KanadeRussell_Staircases, Kursungoz19}:
\begin{align*}
  \sum_{n_1,n_2 \geq 0} \frac{q^{2 n_1^2 + 6 n_1 n_2 + 6 n_2^2}}{(q;q)_{n_1} (q^3;q^3)_{n_2}} = (-q^2,-q^3,-q^4,-q^6; q^6)_\infty.
\end{align*}

\item $A = \begin{bsmallmatrix} 3 & 1 \\ 4 & 2 \end{bsmallmatrix}$.
The modularity follows from the G{\"o}llnitz-Gordon identities in the form presented by Kur{\c{s}}ung{\"o}z \cite[(21) and (22)]{Kursungoz19}:
\begin{align*}
  \sum_{n_1,n_2 \geq 0}
    \frac{q^{\frac{3}{2} n_1^2 + 4 n_1 n_2 + 4 n_2^2 - \frac{1}{2}n_1}}{(q;q)_{n_1} (q^4;q^4)_{n_2}} &= \frac{1}{(q,q^4,q^7;q^8)_\infty}, \\
  \sum_{n_1,n_2 \geq 0}
    \frac{q^{\frac{3}{2} n_1^2 + 4 n_1 n_2 + 4 n_2^2 + \frac{3}{2}n_1 + 4n_2}}{(q;q)_{n_1} (q^4;q^4)_{n_2}} &= \frac{1}{(q^3,q^4,q^5;q^8)_\infty}.
\end{align*}

\item $A = \begin{bsmallmatrix} 1 & -1/2 \\ -1 & 1 \end{bsmallmatrix}$ or $\begin{bsmallmatrix} 1 & -3/2 \\ -1 & 1 \end{bsmallmatrix}$.
If $b = \begin{bsmallmatrix} 0 \\ 0 \end{bsmallmatrix}$, the Nahm sum coincides, up to a multiplication by a product of the dedekind eta function, 
with a certain character of the integrable highest weight module $L(2 \Lambda_0)$ of the affine Lie algebra of type $D_3^{(2)}$ or $D_4^{(3)}$.
Such a character formula was conjectured in \cite{HKOTT} and recently proved in \cite{okado_parafermionic_2022}.
Then the modularity follows from the result by Kac and Peterson \cite{KacPet}.

\end{itemize}

\afterpage{
  \clearpage
{\small
\begin{paracol}{3}
\begin{tblr}{  
  colspec = {Q[l,m,mode=math]Q[c,m,mode=math]Q[r,m,mode=math]},
  column{2} = {colsep+=-5pt}, 
  } 
  \hline
  A & b & c \\ \hline
  \SetCell[r=2]{m}{\begin{pmatrix}
    1 & 1 & 0 \\ 
    1 & 2 & 1 \\ 
    0 & 2 & 4 \\ 
    \end{pmatrix}}
    &\begin{bsmallmatrix}1/2 \\ 1 \\ 0 \\ \end{bsmallmatrix}&\frac{7}{120}\\
    &\begin{bsmallmatrix}1/2 \\ 1 \\ 4 \\ \end{bsmallmatrix}&\frac{103}{120}\\
    \hline
    \SetCell[r=3]{m}{\begin{pmatrix}
    1 & 1 & 0 \\ 
    1 & 3 & 1 \\ 
    0 & 2 & 2 \\ 
    \end{pmatrix}}
    &\begin{bsmallmatrix}1/2 \\ -1 \\ 0 \\ \end{bsmallmatrix}&\frac{3}{56}\\
    &\begin{bsmallmatrix}1/2 \\ 0 \\ 0 \\ \end{bsmallmatrix}&-\frac{1}{56}\\
    &\begin{bsmallmatrix}1/2 \\ 1 \\ 2 \\ \end{bsmallmatrix}&\frac{19}{56}\\
    \hline
    \SetCell[r=5]{m}{\begin{pmatrix}
    1 & 1 & 1 \\ 
    1 & 3 & 2 \\ 
    2 & 4 & 4 \\ 
    \end{pmatrix}}
    &\begin{bsmallmatrix}0 \\ -1/2 \\ 0 \\ \end{bsmallmatrix}&-\frac{1}{32}\\
    &\begin{bsmallmatrix}1/2 \\ -3/2 \\ -2 \\ \end{bsmallmatrix}&\frac{1}{8}\\
    &\begin{bsmallmatrix}1/2 \\ -1/2 \\ 0 \\ \end{bsmallmatrix}&0\\
    &\begin{bsmallmatrix}1/2 \\ 1/2 \\ 2 \\ \end{bsmallmatrix}&\frac{1}{8}\\
    &\begin{bsmallmatrix}1 \\ 1/2 \\ 2 \\ \end{bsmallmatrix}&\frac{7}{32}\\
    \hline
    \SetCell[r=5]{m}{\begin{pmatrix}
    1 & 1 & 0 \\ 
    1 & 4 & 1 \\ 
    0 & 2 & 1 \\ 
    \end{pmatrix}}
    &\begin{bsmallmatrix}1/2 \\ -2 \\ 1 \\ \end{bsmallmatrix}&\frac{7}{24}\\
    &\begin{bsmallmatrix}1/2 \\ 0 \\ 0 \\ \end{bsmallmatrix}&-\frac{1}{48}\\
    &\begin{bsmallmatrix}1/2 \\ 0 \\ 1 \\ \end{bsmallmatrix}&\frac{1}{24}\\
    &\begin{bsmallmatrix}1/2 \\ 2 \\ 1 \\ \end{bsmallmatrix}&\frac{7}{24}\\
    &\begin{bsmallmatrix}1/2 \\ 2 \\ 2 \\ \end{bsmallmatrix}&\frac{23}{48}\\
    \hline
    \SetCell[r=3]{m}{\begin{pmatrix}
    1 & 1 & 0 \\ 
    1 & 4 & 2 \\ 
    0 & 4 & 4 \\ 
    \end{pmatrix}}
    &\begin{bsmallmatrix}1/2 \\ 0 \\ 0 \\ \end{bsmallmatrix}&-\frac{1}{168}\\
    &\begin{bsmallmatrix}1/2 \\ 0 \\ 2 \\ \end{bsmallmatrix}&\frac{47}{168}\\
    &\begin{bsmallmatrix}1/2 \\ 2 \\ 4 \\ \end{bsmallmatrix}&\frac{143}{168}\\
    \hline
    \SetCell[r=2]{m}{\begin{pmatrix}
    1 & 1 & 1 \\ 
    1 & 5 & 4 \\ 
    2 & 8 & 8 \\ 
    \end{pmatrix}}
    &\begin{bsmallmatrix}1/2 \\ -1/2 \\ 0 \\ \end{bsmallmatrix}&\frac{1}{120}\\
    &\begin{bsmallmatrix}1/2 \\ 3/2 \\ 4 \\ \end{bsmallmatrix}&\frac{49}{120}\\
    \hline
    \SetCell[r=2]{m}{\begin{pmatrix}
    1 & 1 & 0 \\ 
    1 & 8 & 1 \\ 
    0 & 2 & 1 \\ 
    \end{pmatrix}}
    &\begin{bsmallmatrix}1/2 \\ 0 \\ 1 \\ \end{bsmallmatrix}&\frac{7}{120}\\
    &\begin{bsmallmatrix}1/2 \\ 4 \\ 1 \\ \end{bsmallmatrix}&\frac{103}{120}\\
    \hline
    \SetCell[r=3]{m}{\begin{pmatrix}
    1 & 1 & 0 \\ 
    1 & 8 & 2 \\ 
    0 & 4 & 2 \\ 
    \end{pmatrix}}
    &\begin{bsmallmatrix}1/2 \\ 0 \\ 0 \\ \end{bsmallmatrix}&-\frac{1}{168}\\
    &\begin{bsmallmatrix}1/2 \\ 2 \\ 0 \\ \end{bsmallmatrix}&\frac{47}{168}\\
    &\begin{bsmallmatrix}1/2 \\ 4 \\ 2 \\ \end{bsmallmatrix}&\frac{143}{168}\\
    \hline
  \end{tblr}
  \switchcolumn \hspace*{-10pt} \hfill
  \begin{tblr}{  
    colspec = {Q[l,m,mode=math]Q[c,m,mode=math]Q[r,m,mode=math]},
    column{2} = {colsep+=-10pt}, 
    } 
    \hline
    A & b & c \\ \hline
    \SetCell[r=9]{m}{\begin{pmatrix}
    2 & 1 & 1 \\ 
    1 & 2 & 1 \\ 
    2 & 2 & 2 \\ 
    \end{pmatrix}}
    &\begin{bsmallmatrix}-1 \\ 0 \\ 1 \\ \end{bsmallmatrix}&\frac{1}{8}\\
    &\begin{bsmallmatrix}-1/2 \\ 0 \\ 0 \\ \end{bsmallmatrix}&-\frac{1}{32}\\
    &\begin{bsmallmatrix}0 \\ -1 \\ 1 \\ \end{bsmallmatrix}&\frac{1}{8}\\
    &\begin{bsmallmatrix}0 \\ -1/2 \\ 0 \\ \end{bsmallmatrix}&-\frac{1}{32}\\
    &\begin{bsmallmatrix}0 \\ 0 \\ 1 \\ \end{bsmallmatrix}&0\\
    &\begin{bsmallmatrix}0 \\ 1 \\ 1 \\ \end{bsmallmatrix}&\frac{1}{8}\\
    &\begin{bsmallmatrix}1 \\ 0 \\ 1 \\ \end{bsmallmatrix}&\frac{1}{8}\\
    &\begin{bsmallmatrix}1/2 \\ 1 \\ 2 \\ \end{bsmallmatrix}&\frac{7}{32}\\
    &\begin{bsmallmatrix}1 \\ 1/2 \\ 2 \\ \end{bsmallmatrix}&\frac{7}{32}\\
    \hline
    \SetCell[r=2]{m}{\begin{pmatrix}
    2 & 1 & 1 \\ 
    1 & 4 & 1 \\ 
    2 & 2 & 2 \\ 
    \end{pmatrix}}
    &\begin{bsmallmatrix}0 \\ 0 \\ 1 \\ \end{bsmallmatrix}&\frac{1}{120}\\
    &\begin{bsmallmatrix}0 \\ 2 \\ 1 \\ \end{bsmallmatrix}&\frac{49}{120}\\
    \hline
    \SetCell[r=5]{m}{\begin{pmatrix}
    2 & 2 & 1 \\ 
    2 & 4 & 2 \\ 
    2 & 4 & 3 \\ 
    \end{pmatrix}}
    &\begin{bsmallmatrix}-1 \\ -1 \\ -1 \\ \end{bsmallmatrix}&\frac{1}{48}\\
    &\begin{bsmallmatrix}0 \\ 0 \\ 0 \\ \end{bsmallmatrix}&-\frac{1}{24}\\
    &\begin{bsmallmatrix}0 \\ 0 \\ 1 \\ \end{bsmallmatrix}&\frac{1}{48}\\
    &\begin{bsmallmatrix}0 \\ 1 \\ 2 \\ \end{bsmallmatrix}&\frac{5}{24}\\
    &\begin{bsmallmatrix}1 \\ 2 \\ 3 \\ \end{bsmallmatrix}&\frac{25}{48}\\
    \hline
    \SetCell[r=4]{m}{\begin{pmatrix}
    1 & 1 & 1/2 \\ 
    1 & 2 & 1 \\ 
    1 & 2 & 3/2 \\ 
    \end{pmatrix}}
    &\begin{bsmallmatrix}0 \\ 0 \\ 0 \\ \end{bsmallmatrix}&-\frac{1}{18}\\
    &\begin{bsmallmatrix}0 \\ 0 \\ 1 \\ \end{bsmallmatrix}&0\\
    &\begin{bsmallmatrix}0 \\ 1 \\ 1 \\ \end{bsmallmatrix}&\frac{1}{9}\\
    &\begin{bsmallmatrix}1 \\ 1 \\ 2 \\ \end{bsmallmatrix}&\frac{5}{18}\\
    \hline
    \SetCell[r=4]{m}{\begin{pmatrix}
    1 & 0 & 1/2 \\ 
    0 & 2 & 1 \\ 
    1 & 2 & 2 \\ 
    \end{pmatrix}}
    &\begin{bsmallmatrix}0 \\ 0 \\ 0 \\ \end{bsmallmatrix}&-\frac{5}{88}\\
    &\begin{bsmallmatrix}0 \\ 0 \\ 1 \\ \end{bsmallmatrix}&-\frac{1}{88}\\
    &\begin{bsmallmatrix}0 \\ 1 \\ 2 \\ \end{bsmallmatrix}&\frac{19}{88}\\
    &\begin{bsmallmatrix}1 \\ 1 \\ 2 \\ \end{bsmallmatrix}&\frac{35}{88}\\
    \hline
  \end{tblr}
  \switchcolumn \hspace*{-10pt} \hfill
  \begin{tblr}{  
    colspec = {Q[l,m,mode=math]Q[c,m,mode=math]Q[r,m,mode=math]},
    column{2} = {colsep+=-10pt}, 
    } 
    \hline
    A & b & c \\ \hline
    \SetCell[r=2]{m}{\begin{pmatrix}
    2 & 2 & 2 \\ 
    2 & 5 & 4 \\ 
    4 & 8 & 8 \\ 
    \end{pmatrix}}
    &\begin{bsmallmatrix}0 \\ -1/2 \\ 0 \\ \end{bsmallmatrix}&-\frac{1}{42}\\
    &\begin{bsmallmatrix}0 \\ 1/2 \\ 2 \\ \end{bsmallmatrix}&\frac{5}{42}\\
    &\begin{bsmallmatrix}1 \\ 3/2 \\ 4 \\ \end{bsmallmatrix}&\frac{17}{42}\\
    \hline
    \SetCell[r=4]{m}{\begin{pmatrix}
    1 & 1 & 0 \\ 
    1 & 3 & 1/2 \\ 
    0 & 1 & 3/2 \\ 
    \end{pmatrix}}
    &\begin{bsmallmatrix}1/2 \\ -1/2 \\ 1/2 \\ \end{bsmallmatrix}&\frac{1}{40}\\
    &\begin{bsmallmatrix}1/2 \\ 1/2 \\ -1/2 \\ \end{bsmallmatrix}&\frac{1}{40}\\
    &\begin{bsmallmatrix}1/2 \\ 1/2 \\ 3/2 \\ \end{bsmallmatrix}&\frac{9}{40}\\
    &\begin{bsmallmatrix}1/2 \\ 3/2 \\ 1/2 \\ \end{bsmallmatrix}&\frac{9}{40}\\
    \hline
    \SetCell[r=3]{m}{\begin{pmatrix}
    1 & 1 & 0 \\ 
    1 & 4 & 1 \\ 
    0 & 2 & 3/2 \\ 
    \end{pmatrix}}
    &\begin{bsmallmatrix}1/2 \\ 0 \\ -1 \\ \end{bsmallmatrix}&\frac{3}{56}\\
    &\begin{bsmallmatrix}1/2 \\ 0 \\ 0 \\ \end{bsmallmatrix}&-\frac{1}{56}\\
    &\begin{bsmallmatrix}1/2 \\ 2 \\ 1 \\ \end{bsmallmatrix}&\frac{19}{56}\\
    \hline
    \SetCell[r=3]{m}{\begin{pmatrix}
    3 & 2 & 2 \\ 
    2 & 4 & 2 \\ 
    4 & 4 & 4 \\ 
    \end{pmatrix}}
    &\begin{bsmallmatrix}-1/2 \\ 0 \\ 0 \\ \end{bsmallmatrix}&-\frac{1}{42}\\
    &\begin{bsmallmatrix}-1/2 \\ 1 \\ 0 \\ \end{bsmallmatrix}&\frac{5}{42}\\
    &\begin{bsmallmatrix}1/2 \\ 2 \\ 2 \\ \end{bsmallmatrix}&\frac{17}{42}\\
    \hline
    \SetCell[r=3]{m}{\begin{pmatrix}
    3/2 & 1 & 1 \\ 
    1 & 3 & 2 \\ 
    2 & 4 & 4 \\ 
    \end{pmatrix}}
    &\begin{bsmallmatrix}-1/2 \\ -1/2 \\ 0 \\ \end{bsmallmatrix}&\frac{1}{168}\\
    &\begin{bsmallmatrix}0 \\ -1/2 \\ 0 \\ \end{bsmallmatrix}&-\frac{5}{168}\\
    &\begin{bsmallmatrix}1/2 \\ 1/2 \\ 2 \\ \end{bsmallmatrix}&\frac{25}{168}\\
    \hline
    \SetCell[r=3]{m}{\begin{pmatrix}
    2 & 1 & 1 \\ 
    1 & 5/2 & 3/2 \\ 
    2 & 3 & 3 \\ 
    \end{pmatrix}}
    &\begin{bsmallmatrix}0 \\ -1 \\ -1 \\ \end{bsmallmatrix}&\frac{1}{168}\\
    &\begin{bsmallmatrix}0 \\ -1/2 \\ 0 \\ \end{bsmallmatrix}&-\frac{5}{168}\\
    &\begin{bsmallmatrix}1 \\ 0 \\ 1 \\ \end{bsmallmatrix}&\frac{25}{168}\\
    \hline
\end{tblr}
\end{paracol}
}
{\centering \captionof{table}{Symmetrizable modular candidates with $d = (1,1,2)$.} \label{table:1,1,2}}
\clearpage
}

\afterpage{
  \clearpage
{\small
\begin{paracol}{3}
\begin{tblr}{
  colspec = {Q[l,m,mode=math]Q[c,m,mode=math]Q[r,m,mode=math]},
  column{2} = {colsep+=-13pt}, 
  } 
  \hline
  A & b & c \\ \hline
  \SetCell[r=4]{m}{\begin{pmatrix}
    1 & 0 & 1 \\ 
    0 & 2 & 2 \\ 
    1/2 & 1 & 2 \\ 
    \end{pmatrix}}
    &\begin{bsmallmatrix}0 \\ 0 \\ 0 \\ \end{bsmallmatrix}&-\frac{7}{88}\\
    &\begin{bsmallmatrix}1 \\ 0 \\ 0 \\ \end{bsmallmatrix}&\frac{1}{88}\\
    &\begin{bsmallmatrix}1 \\ 0 \\ 1 \\ \end{bsmallmatrix}&\frac{9}{88}\\
    &\begin{bsmallmatrix}1 \\ 2 \\ 2 \\ \end{bsmallmatrix}&\frac{49}{88}\\
    \hline
    \SetCell[r=3]{m}{\begin{pmatrix}
    1 & 1 & 1 \\ 
    1 & 2 & 2 \\ 
    1/2 & 1 & 2 \\ 
    \end{pmatrix}}
    &\begin{bsmallmatrix}0 \\ 0 \\ -1/2 \\ \end{bsmallmatrix}&-\frac{3}{56}\\
    &\begin{bsmallmatrix}1 \\ 0 \\ -1/2 \\ \end{bsmallmatrix}&\frac{1}{56}\\
    &\begin{bsmallmatrix}1 \\ 2 \\ 1/2 \\ \end{bsmallmatrix}&\frac{9}{56}\\
    \hline
    \SetCell[r=2]{m}{\begin{pmatrix}
    1 & 1 & 0 \\ 
    1 & 4 & 4 \\ 
    0 & 2 & 3 \\ 
    \end{pmatrix}}
    &\begin{bsmallmatrix}1 \\ 0 \\ -1/2 \\ \end{bsmallmatrix}&\frac{1}{48}\\
    &\begin{bsmallmatrix}1 \\ 4 \\ 3/2 \\ \end{bsmallmatrix}&\frac{25}{48}\\
    \hline
    \SetCell[r=3]{m}{\begin{pmatrix}
    2 & 1 & 0 \\ 
    1 & 2 & 2 \\ 
    0 & 1 & 2 \\ 
    \end{pmatrix}}
    &\begin{bsmallmatrix}0 \\ 1 \\ 0 \\ \end{bsmallmatrix}&-\frac{1}{36}\\
    &\begin{bsmallmatrix}2 \\ 1 \\ 0 \\ \end{bsmallmatrix}&\frac{11}{36}\\
    &\begin{bsmallmatrix}2 \\ 3 \\ 1 \\ \end{bsmallmatrix}&\frac{23}{36}\\
    \hline
    \SetCell[r=10]{m}{\begin{pmatrix}
    2 & 2 & 2 \\ 
    2 & 4 & 4 \\ 
    1 & 2 & 3 \\ 
    \end{pmatrix}}
    &\begin{bsmallmatrix}-1 \\ -2 \\ -3/2 \\ \end{bsmallmatrix}&\frac{1}{24}\\
    &\begin{bsmallmatrix}-1 \\ 0 \\ 1/2 \\ \end{bsmallmatrix}&\frac{1}{24}\\
    &\begin{bsmallmatrix}0 \\ -1 \\ -1 \\ \end{bsmallmatrix}&-\frac{1}{48}\\
    &\begin{bsmallmatrix}0 \\ 0 \\ -1/2 \\ \end{bsmallmatrix}&-\frac{1}{24}\\
    &\begin{bsmallmatrix}0 \\ 0 \\ 0 \\ \end{bsmallmatrix}&-\frac{5}{96}\\
    &\begin{bsmallmatrix}0 \\ 2 \\ 3/2 \\ \end{bsmallmatrix}&\frac{7}{24}\\
    &\begin{bsmallmatrix}1 \\ 0 \\ -1/2 \\ \end{bsmallmatrix}&\frac{1}{24}\\
    &\begin{bsmallmatrix}1 \\ 0 \\ 1/2 \\ \end{bsmallmatrix}&\frac{1}{24}\\
    &\begin{bsmallmatrix}2 \\ 2 \\ 1/2 \\ \end{bsmallmatrix}&\frac{7}{24}\\
    &\begin{bsmallmatrix}3 \\ 4 \\ 5/2 \\ \end{bsmallmatrix}&\frac{25}{24}\\
    \hline
    \SetCell[r=2]{m}{\begin{pmatrix}
    4 & 4 & 2 \\ 
    4 & 5 & 2 \\ 
    1 & 1 & 1 \\ 
    \end{pmatrix}}
    &\begin{bsmallmatrix}0 \\ 0 \\ 1/2 \\ \end{bsmallmatrix}&-\frac{1}{40}\\
    &\begin{bsmallmatrix}4 \\ 4 \\ 1/2 \\ \end{bsmallmatrix}&\frac{31}{40}\\
    \hline
  \end{tblr} 
  \switchcolumn  \hspace*{-15pt} \hfill
  \begin{tblr}{
    colspec = {Q[l,m,mode=math]Q[c,m,mode=math]Q[r,m,mode=math]},
    column{2} = {colsep+=-13pt}, 
    } 
    \hline
    A & b & c \\ \hline
    \SetCell[r=3]{m}{\begin{pmatrix}
    1 & 1 & 1 \\ 
    1 & 2 & 2 \\ 
    1/2 & 1 & 3/2 \\ 
    \end{pmatrix}}
    &\begin{bsmallmatrix}0 \\ 0 \\ 0 \\ \end{bsmallmatrix}&-\frac{5}{72}\\
    &\begin{bsmallmatrix}1 \\ 0 \\ 1/2 \\ \end{bsmallmatrix}&\frac{1}{72}\\
    &\begin{bsmallmatrix}1 \\ 2 \\ 3/2 \\ \end{bsmallmatrix}&\frac{25}{72}\\
    \hline
    \SetCell[r=4]{m}{\begin{pmatrix}
    1 & 0 & 1 \\ 
    0 & 2 & 2 \\ 
    1/2 & 1 & 5/2 \\ 
    \end{pmatrix}}
    &\begin{bsmallmatrix}0 \\ 0 \\ -1/2 \\ \end{bsmallmatrix}&-\frac{5}{72}\\
    &\begin{bsmallmatrix}1 \\ 0 \\ -1 \\ \end{bsmallmatrix}&\frac{1}{8}\\
    &\begin{bsmallmatrix}1 \\ 0 \\ 0 \\ \end{bsmallmatrix}&\frac{1}{72}\\
    &\begin{bsmallmatrix}1 \\ 2 \\ 1 \\ \end{bsmallmatrix}&\frac{25}{72}\\
    \hline
    \SetCell[r=3]{m}{\begin{pmatrix}
    1 & 1 & 1 \\ 
    1 & 3 & 3 \\ 
    1/2 & 3/2 & 5/2 \\ 
    \end{pmatrix}}
    &\begin{bsmallmatrix}0 \\ -1 \\ -1 \\ \end{bsmallmatrix}&-\frac{1}{40}\\
    &\begin{bsmallmatrix}1 \\ -2 \\ -3/2 \\ \end{bsmallmatrix}&\frac{9}{40}\\
    &\begin{bsmallmatrix}1 \\ 0 \\ -1/2 \\ \end{bsmallmatrix}&\frac{1}{40}\\
    \hline
    \SetCell[r=10]{m}{\begin{pmatrix}
    3/2 & 1/2 & 1 \\ 
    1/2 & 3/2 & 1 \\ 
    1/2 & 1/2 & 1 \\ 
    \end{pmatrix}}
    &\begin{bsmallmatrix}-3/2 \\ -1/2 \\ 1/2 \\ \end{bsmallmatrix}&\frac{1}{8}\\
    &\begin{bsmallmatrix}-1/2 \\ -3/2 \\ 1/2 \\ \end{bsmallmatrix}&\frac{1}{8}\\
    &\begin{bsmallmatrix}-1/2 \\ 1/2 \\ 0 \\ \end{bsmallmatrix}&-\frac{1}{32}\\
    &\begin{bsmallmatrix}-1/2 \\ 1/2 \\ 1/2 \\ \end{bsmallmatrix}&0\\
    &\begin{bsmallmatrix}1/2 \\ -1/2 \\ 0 \\ \end{bsmallmatrix}&-\frac{1}{32}\\
    &\begin{bsmallmatrix}1/2 \\ -1/2 \\ 1/2 \\ \end{bsmallmatrix}&0\\
    &\begin{bsmallmatrix}1/2 \\ 3/2 \\ 1/2 \\ \end{bsmallmatrix}&\frac{1}{8}\\
    &\begin{bsmallmatrix}1/2 \\ 3/2 \\ 1 \\ \end{bsmallmatrix}&\frac{7}{32}\\
    &\begin{bsmallmatrix}3/2 \\ 1/2 \\ 1/2 \\ \end{bsmallmatrix}&\frac{1}{8}\\
    &\begin{bsmallmatrix}3/2 \\ 1/2 \\ 1 \\ \end{bsmallmatrix}&\frac{7}{32}\\
    \hline
    \SetCell[r=2]{m}{\begin{pmatrix}
    4 & 4 & 4 \\ 
    4 & 6 & 6 \\ 
    2 & 3 & 4 \\ 
    \end{pmatrix}}
    &\begin{bsmallmatrix}0 \\ -1 \\ -1 \\ \end{bsmallmatrix}&-\frac{1}{60}\\
    &\begin{bsmallmatrix}2 \\ 1 \\ 0 \\ \end{bsmallmatrix}&\frac{11}{60}\\
    \hline
  \end{tblr}
  \switchcolumn \hspace*{-10pt} \hfill
  \begin{tblr}{
    colspec = {Q[l,m,mode=math]Q[c,m,mode=math]Q[r,m,mode=math]},
    column{2} = {colsep+=-13pt}, 
    } 
    \hline
    A & b & c \\ \hline
    \SetCell[r=6]{m}{\begin{pmatrix}
    3/2 & 1/2 & 1 \\ 
    1/2 & 3/2 & 1 \\ 
    1/2 & 1/2 & 3/2 \\ 
    \end{pmatrix}}
    &\begin{bsmallmatrix}-1/2 \\ 1/2 \\ -1/2 \\ \end{bsmallmatrix}&\frac{1}{168}\\
    &\begin{bsmallmatrix}-1/2 \\ 1/2 \\ 0 \\ \end{bsmallmatrix}&-\frac{5}{168}\\
    &\begin{bsmallmatrix}1/2 \\ -1/2 \\ -1/2 \\ \end{bsmallmatrix}&\frac{1}{168}\\
    &\begin{bsmallmatrix}1/2 \\ -1/2 \\ 0 \\ \end{bsmallmatrix}&-\frac{5}{168}\\
    &\begin{bsmallmatrix}1/2 \\ 3/2 \\ 1/2 \\ \end{bsmallmatrix}&\frac{25}{168}\\
    &\begin{bsmallmatrix}3/2 \\ 1/2 \\ 1/2 \\ \end{bsmallmatrix}&\frac{25}{168}\\
    \hline
    \SetCell[r=6]{m}{\begin{pmatrix}
    3/2 & 1/2 & 2 \\ 
    1/2 & 3/2 & 2 \\ 
    1 & 1 & 4 \\ 
    \end{pmatrix}}
    &\begin{bsmallmatrix}-1/2 \\ 1/2 \\ 0 \\ \end{bsmallmatrix}&-\frac{1}{42}\\
    &\begin{bsmallmatrix}-1/2 \\ 1/2 \\ 1 \\ \end{bsmallmatrix}&\frac{5}{42}\\
    &\begin{bsmallmatrix}1/2 \\ -1/2 \\ 0 \\ \end{bsmallmatrix}&-\frac{1}{42}\\
    &\begin{bsmallmatrix}1/2 \\ -1/2 \\ 1 \\ \end{bsmallmatrix}&\frac{5}{42}\\
    &\begin{bsmallmatrix}1/2 \\ 3/2 \\ 2 \\ \end{bsmallmatrix}&\frac{17}{42}\\
    &\begin{bsmallmatrix}3/2 \\ 1/2 \\ 2 \\ \end{bsmallmatrix}&\frac{17}{42}\\
    \hline
    \SetCell[r=4]{m}{\begin{pmatrix}
    5/2 & 3/2 & 1 \\ 
    3/2 & 5/2 & 1 \\ 
    1/2 & 1/2 & 1 \\ 
    \end{pmatrix}}
    &\begin{bsmallmatrix}-1/2 \\ 1/2 \\ 1/2 \\ \end{bsmallmatrix}&\frac{1}{120}\\
    &\begin{bsmallmatrix}1/2 \\ -1/2 \\ 1/2 \\ \end{bsmallmatrix}&\frac{1}{120}\\
    &\begin{bsmallmatrix}3/2 \\ 5/2 \\ 1/2 \\ \end{bsmallmatrix}&\frac{49}{120}\\
    &\begin{bsmallmatrix}5/2 \\ 3/2 \\ 1/2 \\ \end{bsmallmatrix}&\frac{49}{120}\\
    \hline
    \SetCell[r=6]{m}{\begin{pmatrix}
    5/2 & 3/2 & 2 \\ 
    3/2 & 5/2 & 2 \\ 
    1 & 1 & 2 \\ 
    \end{pmatrix}}
    &\begin{bsmallmatrix}-1/2 \\ 1/2 \\ 0 \\ \end{bsmallmatrix}&-\frac{1}{42}\\
    &\begin{bsmallmatrix}1/2 \\ -1/2 \\ 0 \\ \end{bsmallmatrix}&-\frac{1}{42}\\
    &\begin{bsmallmatrix}1/2 \\ 3/2 \\ 0 \\ \end{bsmallmatrix}&\frac{5}{42}\\
    &\begin{bsmallmatrix}3/2 \\ 1/2 \\ 0 \\ \end{bsmallmatrix}&\frac{5}{42}\\
    &\begin{bsmallmatrix}3/2 \\ 5/2 \\ 1 \\ \end{bsmallmatrix}&\frac{17}{42}\\
    &\begin{bsmallmatrix}5/2 \\ 3/2 \\ 1 \\ \end{bsmallmatrix}&\frac{17}{42}\\
    \hline    
\end{tblr}
\end{paracol}
}
{\centering \captionof{table}{Symmetrizable modular candidates with $d = (2,2,1)$.} \label{table:2,2,1}}
\clearpage
}

Finally, we provide a single comment on Table \ref{table:1,1,2} and \ref{table:2,2,1}.
There is no apparent relationship between the candidates in these two lists, 
which are associated with symmetrizers $d = (1,1,2)$ and $d = (2,2,1)$, except that there are several ``Langlands dual'' pairs: 
\begin{gather*}
  \begin{pmatrix}
    2 & 2 & 1\\
    2 & 4 & 2\\
    2 & 4 & 3
  \end{pmatrix}
  \leftrightarrow
  \begin{pmatrix}
    2 & 2 & 2\\
    2 & 4 & 4\\
    1 & 2 & 3
  \end{pmatrix},\quad
  \begin{pmatrix}
    1 & 1 & 1/2\\
    1 & 2 & 1\\
    1 & 2 & 3/2
  \end{pmatrix}
  \leftrightarrow
  \begin{pmatrix}
    1 & 1 & 1\\
    1 & 2 & 2\\
    1/2 & 1 & 3/2
  \end{pmatrix},\\
  \begin{pmatrix}
    1 & 0 & 1/2\\
    0 & 2 & 1\\
    1 & 2 & 2
  \end{pmatrix}
  \leftrightarrow
  \begin{pmatrix}
    1 & 0 & 1\\
    0 & 2 & 2\\
    1/2 & 1 & 2
  \end{pmatrix}.
\end{gather*}
These pairs consist of matrices that are transposed to each other.
The second (resp. first) pair consists of the inverses of (resp. the twice the inverse of) the Cartan matrices of type $B_3$ and $C_3$.
The third pair is more mysterious. We expect that two Nahm sums for the third pair are related by the modular $S$-transformation.
More precisely, according to numerical experiments, we conjecture that the following modular transformation formulas hold.
For matrices
$A = \begin{bsmallmatrix}
  1 & 0 & 1/2\\
  0 & 2 & 1\\
  1 & 2 & 2
\end{bsmallmatrix}$ and
$A^{\vee} = \begin{bsmallmatrix}
  1 & 0 & 1\\
  0 & 2 & 2\\
  1/2 & 1 & 2
\end{bsmallmatrix}$ with symmetrizers $d = (1,1,2)$ and $d^{\vee}=(2,2,1)$, respectively,  
we define vector-valued functions
\begin{align*}
  g (\tau) = 
  \begin{+pmatrix}[rows={l}]
    f_{A, \begin{bsmallmatrix} 0\\0\\0 \end{bsmallmatrix}, -5/88; \, 0} (q) \\
    f_{A, \begin{bsmallmatrix} 0\\0\\1 \end{bsmallmatrix}, -1/88; \, 1} (q) \\
    f_{A, \begin{bsmallmatrix} 0\\1\\1 \end{bsmallmatrix}, 7/88; \, 1} (q) + 
    q f_{A, \begin{bsmallmatrix} 1\\2\\3 \end{bsmallmatrix}, 7/88; \, 1} (q) \\
    f_{A, \begin{bsmallmatrix} 0\\1\\2 \end{bsmallmatrix}, 19/88; \, 0} (q) \\
    f_{A, \begin{bsmallmatrix} 1\\1\\2 \end{bsmallmatrix}, 35/88; \, 0} (q) \\
  \end{+pmatrix},\
  g^{\vee} (\tau) = 
  \begin{+pmatrix}[rows={l}]
    f_{A^{\vee}, \begin{bsmallmatrix} 0\\0\\0 \end{bsmallmatrix}, -7/88} (q) \\
    f_{A^{\vee}, \begin{bsmallmatrix} 0\\0\\1 \end{bsmallmatrix}, 25/88} (q) + 
    f_{A^{\vee}, \begin{bsmallmatrix} 1\\2\\3 \end{bsmallmatrix}, 25/88} (q) \\
    f_{A^{\vee}, \begin{bsmallmatrix} 1\\0\\0 \end{bsmallmatrix}, 1/88} (q) \\
    f_{A^{\vee}, \begin{bsmallmatrix} 1\\0\\1 \end{bsmallmatrix}, 9/88} (q) \\
    f_{A^{\vee}, \begin{bsmallmatrix} 1\\2\\2 \end{bsmallmatrix}, 49/88} (q) \\
  \end{+pmatrix},
\end{align*}
where $f_{A, b, c, d;\, \sigma}(q)$ for $\sigma = 0,1$
is the partial sum of the Nahm sum where the sum is taken over $n \in \mathbb{N}^3$ such that $n_1 \equiv 0 \bmod \sigma$,
and we omit $d$ and $d^{\vee}$ from the notation.
Then we have
\begin{align*}
  g(\tau + 1) = \diag(\zeta_{88}^{-5}, \zeta_{88}^{43},\zeta_{88}^{51},\zeta_{88}^{19},\zeta_{88}^{35}) g(\tau), \quad
  g^{\vee} (\tau + 1) = \diag(\zeta_{88}^{-7}, \zeta_{88}^{25},\zeta_{88}^{1},\zeta_{88}^{9},\zeta_{88}^{49}) g^{\vee} (\tau),
\end{align*}
and we conjecture that
\begin{align}\label{eq:modular S for Langlands pair}
  g (-\frac{1}{\tau}) = S\ g^{\vee} (\frac{\tau}{2}), \quad
  g^{\vee} (-\frac{1}{\tau}) =
  2 S\ g (\frac{\tau}{2}),
\end{align}
where
\begin{align*}
  S=
  \begin{pmatrix}
    \alpha_{5} & \alpha_{4} & \alpha_{3} & \alpha_{2} & \alpha_{1} \\
    \alpha_{4} & \alpha_{1} & -\alpha_{2} & -\alpha_{5} & -\alpha_{3} \\
    \alpha_{3} & -\alpha_{2} & -\alpha_{4} & \alpha_{1} & \alpha_{5} \\
    \alpha_{2} & -\alpha_{5} & \alpha_{1} & \alpha_{3} & -\alpha_{4} \\
    \alpha_{1} & -\alpha_{3} & \alpha_{5} & -\alpha_{4} & \alpha_{2}
  \end{pmatrix}, \quad
  \alpha_k = \sqrt{\frac{2}{11}} \sin \frac{k\pi}{11}.
\end{align*}

\appendix
\section{Cyclic quantum dilogarithm function}
\label{appendix:cyclic dilog}
The \emph{cyclic quantum dilogarithm function} is defined by
\begin{align}
  \label{eq:cyclic dilog}
  D_\zeta (x) = \prod_{t=1}^{m-1} (1 - \zeta^t x)^t
\end{align}
where $\zeta$ is a primitive $m$th root of unity.
Noting that the function $x \mapsto x - \floor{x}$ is periodic with period $1$,
we have the following expression:
\begin{align}
  D_\zeta (x) = \prod_{t \bmod m} (1 - \zeta^t x)^{m (\frac{t}{m} - \floor*{\frac{t}{m}})}.
\end{align}

\begin{proposition}\label{prop:cyclic dilog change of zeta}
  Let $p, q$ be integers coprime to $m$.
  Suppose that $p > 0$ and $pq \equiv 1 \bmod m$. Then we have
  \begin{align}\label{eq:cyclic dilog change of zeta 1}
    \frac{D_\zeta(x)^p} {D_{\zeta^q}(x)} =
    \left((1-x^m)^{p-1} \prod_{t=1}^{p-1} \frac{1}{(x; \zeta)_{\floor*{\frac{mt}{p}} + 1}} \right)^m.
  \end{align}
  Moreover, we also have 
  \begin{align}\label{eq:cyclic dilog change of zeta 2}
    \frac{D_\zeta(x)^p} {D_{\zeta^q}(x)} \equiv
    \left(\prod_{t=1}^{p-1} \frac{1}{(x; \zeta)_{tq}}\right)^m
    \quad\mod (1 - x^m)^m.
  \end{align}
\end{proposition}
\begin{proof}
  We first see that 
  \begin{align*}
    \frac{D_\zeta(x)^p} {D_{\zeta^q}(x)} =
    \prod_{t=1}^{m-1} (1 - \zeta^t x)^{m \floor*{\frac{t p}{m}}}
  \end{align*}
  since
  \begin{align*}
    \prod_{t=1}^{m-1} (1 - \zeta^t x)^{m(\frac{pt}{m} - \floor*{\frac{pt}{m}})}
    = \prod_{t=1}^{m-1} (1 - \zeta^{qt} x)^{m(\frac{pqt}{m} - \floor*{\frac{pqt}{m}})}
    = \prod_{t=1}^{m-1} (1 - \zeta^{qt} x)^t.
  \end{align*}
  Let $N_r$ for $r = 1, \dots, p$ be the number of integer $t$ such that $0 < t < m-1$ and
  $\floor*{\frac{tp}{m}} \leq r - 1$. The $N_r$ is exactly the number of integer $t$
  such that $0 < t < \frac{mr}{p}$, and thus we have the following formulas:
  \begin{align*}\label{eq:floor tp/m 3}
    N_r = \floor*{\frac{mr}{p}} \quad\text{for $r = 1, \dots, p-1$}, \quad
    N_p = m - 1.
  \end{align*}
  Now \eqref{eq:cyclic dilog change of zeta 1} follows from the following computation:
  \begin{align*}
    \prod_{t=1}^{m-1} (1 - \zeta^t x)^{\floor*{\frac{t p}{m}}} =
    \prod_{r=1}^{p} \frac{(x; \zeta)_{N_r+1}^{r-1}}{(x; \zeta)_{N_{r-1}+1}^{r-1}} &=
    (x; \zeta)_{N_p+1}^{p-1} \prod_{r=1}^{p-1} \frac{1}{(x; \zeta)_{N_r+1}} \\ &=
    (1-x^m)^{p-1} \prod_{t=1}^{p-1} \frac{1}{(x; \zeta)_{\floor*{\frac{mt}{p}} + 1}}.
  \end{align*}
  To prove \eqref{eq:cyclic dilog change of zeta 2}, it suffices to show that the sets $M(r)$ and $M'(r)$ defined by
  \begin{align*}
    M(r) &= \{ t \in [1, p-1] \mid \floor*{\frac{mt}{p}} +1 = r \} \\
    M'(r) &= \{ t \in [1, p-1] \mid tq \equiv r \bmod m \}
  \end{align*}
  have the same cardinality for all $r = 1, \dots, m-1$, where $[1, p-1] \coloneqq \{1, 2, \dots, p-1\}$.
  The equation $tq \equiv r \bmod m$ holds if and only if $t \equiv pr \bmod m$.
  The number of such $t$ is the same as the number of $s \in [1,p-1]$ such that 
  $0 < pr - ms <p$, which is equivalent to the condition $\floor*{\frac{ms}{p}} = r-1$.
\end{proof}

We also see that
\begin{equation}\label{eq:cyclic dilog change of zeta 3}
    \prod_{t=0}^{p-1} \frac{(x;\zeta)_{tq}}{(\zeta^e x;\zeta)_{tq}}  
    = \prod_{t=0}^{p-1} \frac{(x;\zeta)_{e}}{(\zeta^{tq} x;\zeta)_{e}}
    = \frac{(x; \zeta)_k^e}{(x; \zeta^q)_{pe}},
\end{equation}
for any $e \in \mathbb{Z}$,
where the left-hand side is the ratio of the $m$th root of the right-hand side of \eqref{eq:cyclic dilog change of zeta 2}
for $\zeta^e x$ and $x$.
The second equality in \eqref{eq:cyclic dilog change of zeta 3} follows from
\begin{align*}
  \prod_{t=0}^{p-1} (\zeta^{tq} x;\zeta)_{e}
  = \prod_{i=0}^{e-1} \prod_{t=0}^{p-1} (1-\zeta^{tq+i} x)
  = \prod_{t=0}^{pe-1} (1-\zeta^{tq} x)
  = (x; \zeta^q)_{pe}.
\end{align*}

\bibliographystyle{plain}
\footnotesize{\bibliography{../../yyyy}}
\end{document}